\documentclass[11pt,twoside]{article}

\usepackage{amsfonts}
\usepackage{amssymb}
\usepackage{amsmath}
\usepackage{epsfig}
\usepackage{bbm}
\usepackage{mathrsfs}
\usepackage{amsfonts,latexsym,bm}
\usepackage{amsmath,amsthm}
\usepackage{amssymb,amscd}
\usepackage{amsfonts,amsbsy}
\usepackage{fancyhdr,graphicx}
\usepackage[dvips]{psfrag}
\usepackage{indentfirst}
\usepackage[square, comma, sort&compress, numbers]{natbib}
\newtheorem {theorem} {Theorem}[section]
\newtheorem {proposition} [theorem]{Proposition}

\newtheorem {lemma}  [theorem]{Lemma}

\newtheorem {definition} [theorem]{Definition}

\theoremstyle{remark}

\theoremstyle{definition}

\newtheorem {remark} [theorem]{Remark}
\begin{document}
\setlength{\parindent}{4ex}
\setlength{\parskip}{2ex}
\setlength{\oddsidemargin}{12mm}
\setlength{\evensidemargin}{9mm}

\title
{\textsc {Limit cycles by perturbing quadratic isochronous centers inside piecewise smooth polynomial differential systems}}
\begin{figure}[b]
\rule[-0.5ex]{7cm}{0.2pt}\\
\footnotesize $^{*}$Corresponding author. E-mail address:
cenxiuli2010@163.com (X. Cen), lyang@math.tsinghua.edu.cn(L. Yang) and mzhang@math.tsinghua.edu.cn(M. Zhang).
\end{figure}
\author
{{\textsc {Xiuli Cen, Lijun Yang and Meirong Zhang}}\\[2ex]
{\footnotesize\it  Department of Mathematical Sciences, Tsinghua University, Beijing, 100084, P.R.China}}
\date{}
\maketitle {\narrower \small \noindent {\bf Abstract\,\,\,} {In the present paper, we study the number of zeros of the first order Melnikov function for piecewise smooth polynomial differential system, to estimate the number of
limit cycles bifurcated from the period annuli of quadratic isochronous centers, when they are perturbed inside the class of all piecewise smooth polynomial differential systems of degree $n$ with the straight line of discontinuity $x=0$. A sharp upper
bound for the number of zeros of the first order Melnikov functions with respect to quadratic isochronous centers $S_1, S_2$ and $S_3$ is provided. For quadratic isochronous center  $S_4$, we give a rough estimate for the number of zeros of the first order Melnikov function due to its complexity. Furthermore, we
improve the upper bound associated with $S_4$, from $14n+11$ in \cite{LLLZ}, $12n-1$ in \cite{SZ} to $[(5n-5)/2]$, when it is perturbed inside all smooth polynomial differential systems of degree $n$. Besides, some evidence on the equivalence between the first order Melnikov function method and the first order
Averaging method for investigating the number of limit cycles of piecewise smooth polynomial differential systems is found.}

}


\section{Introduction and statement of the main result}

In the qualitative theory of real planar differential systems, one important open problem is the determination of limit cycles.
The study of limit cycles for smooth polynomial differential systems origins from the well-known Hilbert 16th Problem, and has
achieved lots of rich and excellent works, see the survey \cite{L} and the references therein.
Nevertheless, it is still open even for the quadratic cases. As substantial
piecewise smooth differential systems have emerged in control theory, electronic
circuits with switches, and mechanical engineering with impact and dry frictions etc., the investigation of limit cycles for
piecewise smooth differential systems attracts lots of mathematicians' widespread concerns. They attempt to develop the theory
on piecewise smooth differential systems, and generalise the tools for studying the number of limit cycles from smooth differential systems to piecewise smooth differential systems. To the best of our knowledge,  the Melnikov function method \cite{LH,LCZ}
and the Averaging method \cite{LNT} are two main tools extended to study the number of limit cycles for piecewise smooth differential systems.

A center of a real planar polynomial differential system is called an isochronous center
if there exists a neighborhood of which such that all periodic orbits in this neighborhood have the
same period. Owing to its speciality, the isochronous centers have attracted much more attentions, see the survey \cite{C}.

The quadratic polynomial differential systems with an isochronous center were first classified into four kinds in \cite{L} by Loud. Using the notation of \cite{MRT}, we exhibited the four classes of quadratic isochronous centers and their first integrals as follows, see Table \ref{Tab:S}.
\begin{table}[h]
\caption{Quadratic isochronous centers and first integrals}
\vspace{2pt}
\centering
\doublerulesep=0.4pt
\begin{tabular}{cll}
\hline\hline\\[-8pt]
Name &  \quad System & \quad First integral\\[1ex] \hline\\[-8pt]
$S_1$ & \quad $\dot{x}=-y+\frac{1}{2}x^{2}-\frac{1}{2}y^{2}$ &  \quad $H=\frac{x^2+y^2}{1+y}$\\[1ex]
& \quad $\dot{y}=x(1+y)$ &\\[1ex]
$S_2$ & \quad $\dot{x}=-y+x^{2}$ & \quad $H=\frac{x^2+y^2}{(1+y)^2}$\\[1ex]
& \quad $\dot{y}=x(1+y)$ &\\[1ex]
$S_3$ & \quad $\dot{x}=-y+\frac{1}{4}x^{2}$ & \quad $H=\frac{(x^2+4y+8)^2}{1+y}$\\[1ex]
& \quad $\dot{y}=x(1+y)$ &\\[1ex]
$S_4$ & \quad $\dot{x}=-y+2x^{2}-\frac{1}{2}y^{2}$ & \quad $H=\frac{4x^2-2(y+1)^2+1}{(1+y)^4}$\\[1ex]
& \quad $\dot{y}=x(1+y)$ &\\[1ex]
\hline\hline
\end{tabular}
\label{Tab:S}
\end{table}

Studies on the number of limit cycles bifurcated from the period annuli of quadratic isochronous centers, when they are perturbed inside all smooth polynomial  differential systems of degree $n$, have exhibited a relatively complete result. For $n=2$, in \cite{CJ}, by the first order bifurcation,  Chicone and Jacobs  proved that  at most 1 limit cycle bifurcates from the periodic orbits of $S_1$, and at most 2 limit cycles bifurcate from the periodic orbits of $S_2,S_3$ and $S_4$;
Iliev obtained in \cite{I} that the cyclicity of the period annulus around $S_1$ is also 2. Li et al in \cite{LLLZ} presented a linear estimate for the number of limit cycles with respect to the four quadratic isochronous centers for any natural number $n$, where the upper bounds for systems $S_1,S_2$ and $S_3$ are sharp, while the upper bound for system $S_4$ is not.  Shao and Zhao gave an improved upper bound for system $S_4$ in \cite{SZ}.

Currently, researches on the number of limit cycles bifurcated from the period annuli of quadratic isochronous centers, when they are perturbed inside all piecewise smooth polynomial  differential systems of degree $2$, also obtained some results.
Llibre and Mereu used the averaging method of first order to study the number of limit cycles bifurcated from the period annuli of $S_1$
and $S_2$, when they are perturbed inside a class of piecewise smooth quadratic polynomial differential systems \cite{LM} with the straight line of discontinuity $y=0$, and found that
at least 4 and 5 limit cycles can bifurcate from the period annuli of $S_1$ and $S_2$, respectively.
Li and Cen obtained in \cite{LC} that there are at most 4 limit cycles bifurcating from the periodic orbits of $S_3$ by the averaging method of first order and \emph{Chebyshev criterion}, when it is perturbed inside a class of discontinuous quadratic polynomial differential systems with the straight line of discontinuity $x=0$. Cen et al applied the same methods and proved in \cite{CLZ} that there are at most 5 limit cycles bifurcating from the periodic orbits of $S_4$.

In the present paper, we will consider the piecewise smooth polynomial perturbations of degree $n$ for all four quadratic isochronous centers,
and investigate the number of zeros of the first order Melnikov functions associated with these quadratic isochronous centers to study the maximum
number of limit cycles bifurcated from the period annuli. As far as we know, studies on the number of limit cycles for quadratic polynomial differential systems with a center under piecewise smooth polynomial perturbations of degree $n$ are rare, see for instance \cite{LL}.

Suppose that $H=H(x,y)$ is a first integral of the quadratic isochronous center, and $R=R(x,y)$ is the corresponding  integrating factor. We consider the piecewise smooth polynomial perturbations of quadratic isochronous center:
\begin{equation}
\left(\begin{array}{ll}\dot{x}\\[2ex] \dot{y}\end{array}\right)=\label{S}\left\{\begin{array}{ll}
\left(\begin{array}{ll}-\dfrac{H_y}{R}+\varepsilon P^+(x,y)\\
\dfrac{H_x}{R}+\varepsilon Q^+(x,y)\end{array}\right), & \mbox{ $x> 0$,}
\\[2ex]
\left(\begin{array}{ll}-\dfrac{H_y}{R}+\varepsilon P^-(x,y)\\
\dfrac{H_x}{R}+\varepsilon Q^-(x,y)\end{array}\right), & \mbox{ $x< 0$,}
\end{array} \right.
\end{equation}
where $0<|\varepsilon|\ll1$ and $P^{\pm}(x,y),Q^{\pm}(x,y)$ are polynomials in the variables $x$ and $y$ of degree $n$, given by
\begin{equation}\label{PQ}\begin{split}
P^{\pm}(x,y)=\sum_{i+j=0}^{n}a_{ij}^{\pm}x^iy^j,\quad Q^{\pm}(x,y)=\sum_{i+j=0}^{n}b_{ij}^{\pm}x^iy^j.
\end{split}\end{equation}

Adopting the first order Melnikov function method for piecewise smooth integrable non-Hamiltonian systems \cite{LCZ}, we have the following main results.
\begin{theorem}\label{th:S}
Denote the least upper bound for the number of zeros (taking into account their multiplicity) of the first order Melnikov function associated with the quadratic isochronous center $S_i$ by $H_i(n)$, $i=1,2,3,4$. Then
\vspace{-10pt}
\begin{itemize}
\item[(a)]$H_1(0)=1$; $H_1(n)=n+3$ for $n=1,2,3$; and $H_1(n)=2n$ for $n\geq4$;
\item[(b)]$H_2(n)=n+1$ for $n=0,1$; and $H_2(n)=2n+2$ for $n\geq2$;
\item[(c)]$H_3(0)=1$; $H_3(n)=2n+1$ for $n=1,2$; and $H_3(n)=2n+2$  for $n\geq3$;
\item[(d)]$H_4(0)=1$; $H_4(1)=4$; $H_4(2)=7$; $H_4(n)\leq12n+4$ for $n=3,4,5$; and $H_4(n)\leq20n-10-2(1+(-1)^n)$ for $n\geq6$.
\end{itemize}
\vspace{-10pt}
Notice that ``$=$'' denotes the upper bound is sharp.
\end{theorem}
\begin{remark}
(i) In \cite{LC} and \cite{CLZ}, the authors respectively studied the case of $n=2$ when
$a_{00}^{\pm}=b_{00}^{\pm}=0$ for systems $S_3$ and $S_4$ . They proved that
at most 4 and 5 limit cycles can bifurcate from the period annuli of these two quadratic isochronous centers by the averaging method of first order. Compared to the results shown in Theorem \ref{th:S}, the piecewise smooth polynomial perturbations with the constant term can make the
perturbed systems produce at least one more limit cycle.

(ii) In the estimation of number of zeros of Melnikov functions for systems $S_3$ and $S_4$, we find that using the first order Melnikov method
can lead to the same results as those using the first order Averaging method for piecewise smooth polynomial differential systems when $n=2$. Han et al showed the equivalence between the Melnikov method and the Averaging method for studying the number of limit cycles, which are bifurcated
from the period annulus of planar analytic differential systems in
\cite{HRZ}. Here some evidences demonstrate that the equivalence may also hold for piecewise smooth analytic differential systems.

(iii) If $a_{ij}^+=a_{ij}^-$ and $b_{ij}^+=b_{ij}^-$, then the perturbed systems are smooth. Li et al have studied these systems in \cite{LLLZ}, and have proven that:
\begin{theorem}\cite{LLLZ}
The least upper bound for the number of zeros (taking  into account their
multiplicity) of the first Melnikov function (Abelian integral) associated with the system:
\begin{itemize}
\item[(a)]$S_1$ is 0 if $n = 0$; 1 if $n = 1, 2, 3$; $n-2$ for $n\geq4$;
\item[(b)]$S_2$ is 0 if $n = 0, 1$; and $n$ for $n\geq2$.
\item[(c)]$S_3$ is $n$ for all $n\geq0$;
\item[(d)]$S_4$ is $\leq14n + 11$.
\end{itemize}
\end{theorem}
Applying Abelian integrals and complete elliptic integrals of the first and second kinds,  Shao and Zhao give a smaller upper bound on the number of zeros of the first Melnikov function for system $S_4$ in \cite{SZ}. That is, the least upper bound for the number of zeros is $0$, for $n=0$; is not greater than $35$ for $n=1,2,3$; is not greater than  $59$ for $n=4,5,6$; and is not greater than $12n-1$ for $n\geq7$.

We give a further investigation of the upper bound with respect to $S_4$, and obtain a better result as follows.
\begin{theorem}\label{th:S4}
The least upper bound for the number of zeros (taking into account their multiplicity) of the first order Melnikov function associated with the system $S_4$ is $\leq [\frac{5n-5}{2}]$.
\end{theorem}
\end{remark}

As is demonstrated above, the main object of this paper is to provide the least upper bound for the number of zeros of the first order Melnikov functions
with respect to the quadratic isochronous centers $S_1, S_2, S_3$ and $S_4$, when they are perturbed by piecewise smooth polynomials of degree $n$, and to
give an improved upper bound for $S_4$ in \cite{LLLZ,SZ}, when it is perturbed by smooth polynomials of degree $n$. Incidentally, we contrast the results
obtained in \cite{LC} and \cite{CLZ} which deal with the quadratic isochronous centers $S_3$ and $S_4$ respectively by the averaging method of first order,
when $n=2$ and $a_{00}^{\pm}=b_{00}^{\pm}=0$. We find that (i) if the piecewise smooth polynomial perturbations include the constant term, the perturbed
systems could produce at least one more limit cycle, (ii) the first order Melnikov method and the first order averaging method are equivalent in studying
the number of limit cycles bifurcated from the period annuli of the centers.

We exploit the technique than in \cite{LLLZ} to compute
the first order Melnikov functions, that is, we calculate the Melnikov function
through a double integral by using Green's theorem. It has great advantages in obtaining a more accurate expression, and the double integrals are very easy
to compute. For $S_4$, a more precise representation of the  first order Melnikov function than in \cite{LLLZ} is obtained, and thus
a better upper bound for the smooth case can be acquired. Since the Melinikov functions include different kinds of
elementary functions, or elliptic integrals, to obtain a better estimate for the number of zeros of the first order Melnikov
function, we always eliminate these elementary functions step by step orderly. It consists in getting rid of the logarithm
function first, and then eliminating functions which include polynomials as numerators by multiplying nonzero factors and taking derivatives, and thus it suffices to consider the derived function. A useful lemma is proposed to be helpful for determining the exact upper bound for systems $S_1$, $S_2$ and $S_3$, and a better upper bound for system $S_4$.  For small $n$,  the common \emph{Chebyshev criterion} and the properties on extended Chebyshev systems with positive accuracy are used to obtain a sharp upper bound. When the polynomial perturbations are smooth, the Chebyshev property on two-dimensional Fuchsian systems 
also plays a key role in resulting in Theorem \ref{th:S4}.

The present paper is organized as follows.  First, some useful preliminary results are given in Section \ref{sec:pre}. Then, we estimate the number of zeros of the first order Melnikov function for quadratic isochronous centers $S_1, S_2, S_3$ and $S_4$ in Sections \ref{sec:S1}-\ref{sec:S4} respectively, when they are perturbed inside piecewise smooth polynomial differential systems. The proof of Theorem \ref{th:S4} is provided in Section \ref{sec:S4s}, in which an improved result on the number of zeros of the first order Melnikov function for quadratic isochronous center $S_4$ under smooth polynomial perturbations is obtained. Finally, some important proofs and results are given in Appendix  for reference.

\section{Preliminary results}\label{sec:pre}
In this section, we introduce the main method that we will use to study the piecewise smooth polynomial systems \eqref{S}, and
some useful tools and results to estimate the number of zeros of the first Melnikov function.

From Theorem 1 of \cite{LCZ}, the first order Melnikov function with respect to system \eqref{S} is
\begin{equation}\label{M0}
\begin{split}
M(h)=\int_{L_h^+}RP^+\mathrm{d}y-RQ^+\mathrm{d}x+\int_{L_h^-}RP^-\mathrm{d}y-RQ^-\mathrm{d}x,
\end{split}
\end{equation}
where $L_h^+=\{x\geq0|H=h, h\in(h_c,h_s)\}$ and $L_h^-=\{x\leq0|H=h, h\in(h_c,h_s)\}$, see Figure \ref{fig1}. Here $H=h$ is one
periodic orbit of system \eqref{S}$|_{\varepsilon=0}$, and $h_c$ and $h_s$ correspond to the center and the separatrix polycycle, respectively.
Inspired by the idea of \cite{LLLZ}, we will use Green's theorem to compute $M(h)$ through two double integrals.

Let $\mathcal{A}_h$ and $\mathcal{B}_h$ be the two intersection points of $L_h=L_h^+\bigcup L_h^-$ and $y$-axis, and $D_h^+$ and $D_h^-$ be the regions formed by $L_h^+\bigcup\overrightarrow{\mathcal{A}_h\mathcal{B}_h}$ and $L_h^-\bigcup\overrightarrow{\mathcal{B}_h\mathcal{A}_h}$, respectively. Then by Green's theorem, the first order Melnikov function \eqref{M0} can be expressed as
\begin{equation}\label{Mh}
\begin{split}
M(h)=&\iint_{D_h^+}\left[\frac{\partial(RP^+)}{\partial x}+\frac{\partial(RQ^+)}{\partial y}\right]\mathrm{d}x\ \mathrm{d}y-\int_{\overrightarrow{\mathcal{A}_h\mathcal{B}_h}}RP^+\mathrm{d}y-RQ^+\mathrm{d}x\\
&+\iint_{D_h^-}\left[\frac{\partial(RP^-)}{\partial x}+\frac{\partial(RQ^-)}{\partial y}\right]\mathrm{d}x\ \mathrm{d}y-\int_{\overrightarrow{\mathcal{B}_h\mathcal{A}_h}}RP^-\mathrm{d}y-RQ^-\mathrm{d}x.
\end{split}
\end{equation}
In the subsequent sections, we will exploit formula \eqref{Mh} to obtain the specific expressions of the first order Melnikov functions
for the four quadratic isochronous centers.
\begin{figure}[h]
\centering
\includegraphics[width=.45\textwidth]{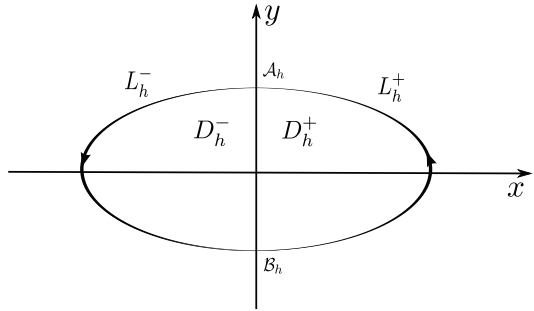}
\caption{\small{The periodic orbit of \eqref{S}$|_{\varepsilon=0}$.}}
\label{fig1}
\end{figure}

For a more complicated function, it is not an easy thing to determine the exact number of its zeros. Here we provide
some effective results to obtain the lower bound and the upper bound of the number of zeros for a more complicated function.
The next result is well known for a lower bound.

\begin{lemma}\label{le:CGP}\cite{CGP}
Consider $n$ linearly independent analytical functions $f_i(x):D\rightarrow\mathbb{R},i=1,2,\cdots,n$,
where $D\subset\mathbb{R}$ is an interval. Suppose that there exists $k\in\{1,2,\cdots,n\}$ such that $f_k(x)$ has constant sign. Then there exists $n$ constants $c_i, i=1,2,\cdots,n$ such that $c_1f_1(x)+
c_2f_2(x)+\cdots+c_nf_n(x)$ has at least $n-1$ simple zeros in $D$.
\end{lemma}

To obtain a better upper bound for the number of zeros of the first order Melnikov
function, we give the following formula, which plays a key role in determining the least upper bound. The proof is putted in Appendix A.1.
\begin{lemma}\label{lem:P}
If $a\neq0$, $p, q\not\in \mathbb{Z}$ and $p+q\in \mathbb{Z}$, then
\begin{equation}\begin{split}
\left(\displaystyle\frac{P_{n}(x)}{x^{p}(a\pm x)^{q}}\right)^{(n+1-(p+q))}=\dfrac{\widetilde P_{n}(x)}{x^{n+1-q}(a\pm x)^{n+1-p}},  \quad n\geq p+q-1,
\end{split}\end{equation}
where $P_n(x)$ and $\widetilde P_{n}(x)$ are polynomials of degree $n$, $f^{(k)}$ denotes the $k$-order derivative of the function $f$.
\end{lemma}

In addition, for small $n$, the theory on \emph{Extended
Complete Chebyshev system} (in short, ECT-system) is useful for an exact upper bound.
Let $\mathcal{F}=(f_1, f_2, \cdots, f_n)$ be an ordered set of $\mathcal{C}^\infty$ functions on $L$.
We call it an ECT-system on $L$ if, for all
$i=1,2,...,n$, any nontrivial linear combination
$\lambda_{1}f_{1}(x)+\lambda_{2}f_{2}(x)+...+\lambda_{i}f_{i}(x)$
has at most $i-1$ isolated zeros on $L$ counted with multiplicities. Moreover, this bound can be reached \cite{GV}.

\begin{lemma}\label{ECT3}
\cite{MV} $\mathcal{F}$ is an ECT-system on L
if, and only if, for each $i=1,2,...,n$,
\begin{equation*}
W_{i}(x)=
\begin{vmatrix}
f_{1}(x) &  f_{2}(x) & \cdots &  f_{i}(x)\\
f_{1}^{\prime}(x) &f_{2}^{\prime}(x)  & \cdots & f_{i}^{\prime}(x) \\
\vdots & \vdots & \ddots & \vdots \\
f_{1}^{(i-1)}(x) & f_{2}^{(i-1)}(x) & \cdots & f_{i}^{(i-1)}(x) \\
  \end{vmatrix}\neq{0}, \quad \mbox{$x\in L$} .
\end{equation*}
\end{lemma}

If the order set $\mathcal{F}$ is not an ECT-system, i.e., some Wronskian determinant has zeros, then the following two lemmas are powerful. The first one
provides an sharp upper bound, while the second one provides a lower bound, which extends Lemma \ref{le:CGP}.
\begin{lemma}\label{le:NT}\cite{NT}
Let $\mathcal{F}$ be an ordered set of $\mathcal{C}^\infty$ functions on $[a, b]$. Assume that all the Wronskians are nonvanishing except $W_n(x)$, which has exactly one zero on $(a, b)$ and this zero is simple. Then $Z(F)=n$ and for any configuration of $m\leq n$ zeros there exists an element in $\mathrm{Span}(\mathcal F)$ realizing it, where $Z(\mathcal F)=n$ denotes the maximum number of zeros counting multiplicity that any nontrivial function $F\in\mathrm{Span}(\mathcal F)$ can have.
\end{lemma}

\begin{lemma}\label{le:NT2}\cite{NT}
Let $\mathcal{F}$ be an ordered set of real $\mathcal{C}^\infty$ functions on $(a, b)$ satisfying that all the Wronskians are nonvanishing
except $W_{n-1}(x)$ and $W_n(x)$, such that there exists $\xi\in(a, b)$ with $W_{n-1}(\xi)\neq0$.
If $W_n(\xi)=0$ and $W'_n(\xi)\neq0$, then for each configuration of $m\leq n$ zeros, taking account their multiplicity, there exists $F\in\mathrm{Span}(\mathcal F)$ with this configuration of zeros.
\end{lemma}

Some results on Two-dimensional Fuchsian systems and the Chebyshev property in \cite{GI} are given in Appendix A.4 for reference, which will be
used to obtain an improved upper bound on quadratic isochronous center $S_4$, when it is perturbed inside all smooth polynomial differential systems of degree $n$.
\section{Zeros of $M(h)$ for system $S_1$}\label{sec:S1}
Consider the piecewise smooth polynomial perturbations of degree $n$ of system $S_1$:
\begin{equation}
\left(\begin{array}{ll}\dot{x}\\[2ex] \dot{y}\end{array}\right)=\label{S1}\left\{\begin{array}{ll}
\left(\begin{array}{ll}-y+\frac{1}{2}x^{2}-\frac{1}{2}y^{2}+\frac{\varepsilon }{2}P^+(x,y)\\
x(1+y)+\frac{\varepsilon }{2}Q^+(x,y)\end{array}\right), & \mbox{ $x> 0$,}
\\[2ex]
\left(\begin{array}{ll}-y+\frac{1}{2}x^{2}-\frac{1}{2}y^{2}+\frac{\varepsilon }{2}P^-(x,y)\\
x(1+y)+\frac{\varepsilon }{2}Q^-(x,y)\end{array}\right), & \mbox{ $x< 0$,}
\end{array} \right.
\end{equation}
where $P^{\pm}(x,y)$ and $Q^{\pm}(x,y)$ are given by \eqref{PQ}. For ${\varepsilon=0}$, a first integral of system \eqref{S1} is
\[
H=\frac{x^2+y^2}{1+y},
\]
and the integrating factor is $R=\frac{2}{(1+y)^2}$. Here $L_h^+=\{x\geq0|H=h,h>0\}$ and $L_h^-=\{x\leq0|H=h,h>0\}$ are the right part and the left part of the periodic orbits surrounding the origin, respectively. $\mathcal{A}_h=(0,\alpha(h))$ and $\mathcal{B}_h=(0,\beta(h))$, where
\begin{equation}
\alpha(h)=\frac{h+\sqrt{h(4+h)}}{2}, \quad \beta(h)=\frac{h-\sqrt{h(4+h)}}{2}.
\end{equation}

\subsection{Expression of $M(h)$}\label{sub:S11}

This subsection is devoted to obtaining the expression of the first order Melnikov function of system \eqref{S1}. By \eqref{Mh},
\begin{equation}\label{M}
\begin{split}
M(h)=M^+(h)+M^-(h),
\end{split}
\end{equation}
where
\begin{equation*}\label{M1}
\begin{split}
M^+(h)=&\iint_{D_h^+}\left[\frac{\partial}{\partial x}\left(\frac{P^+}{(1+y)^2}\right)+\frac{\partial}{\partial y}\left(\frac{Q^+}{(1+y)^2}\right)\right]\mathrm{d}x\ \mathrm{d}y+\int_{\beta(h)}^{\alpha(h)}\frac{P^+(0,y)}{(1+y)^2}\mathrm{d}y,\\
M^-(h)=&\iint_{D_h^-}\left[\frac{\partial}{\partial x}\left(\frac{P^-}{(1+y)^2}\right)+\frac{\partial}{\partial y}\left(\frac{Q^-}{(1+y)^2}\right)\right]\mathrm{d}x\ \mathrm{d}y-\int_{\beta(h)}^{\alpha(h)}\frac{P^-(0,y)}{(1+y)^2}\mathrm{d}y.\\
\end{split}
\end{equation*}
To acquire the expression of $M(h)$, it suffices to compute $M^+(h)$, and $M^-(h)$ can be obtained in the same way.
\begin{equation*}\label{M^+1}
\begin{split}
M^+(h)=&\iint_{D_h^+}\left[\frac{1}{(1+y)^2}\left(\frac{\partial P^+ }{\partial x}+\frac{\partial Q^+ }{\partial y}\right)-\frac{2Q^+}{(1+y)^3}\right]\mathrm{d}x\ \mathrm{d}y+\int_{\beta(h)}^{\alpha(h)}\frac{P^+(0,y)}{(1+y)^2}\mathrm{d}y\\
=&\iint_{D_h^+}\left[\frac{1}{(1+y)^2}\sum_{i+j\leq n}\left(ia^+_{ij}x^{i-1}y^j+jb^+_{ij}x^{i}y^{j-1}\right)-\frac{2}{(1+y)^3}\sum_{i+j\leq n}b^+_{ij}x^{i}y^j\right]\mathrm{d}x\ \mathrm{d}y\\
&+\int_{\beta(h)}^{\alpha(h)}\frac{1}{(1+y)^2}\sum_{j=0}^na^+_{0j}y^{j}\mathrm{d}y\\
=&\iint_{D_h^+}\left(\frac{1}{(1+y)^2}\sum_{i+j\leq n}ia^+_{ij}x^{i-1}y^j+\frac{1}{(1+y)^3}\sum_{i+j\leq n}b^+_{ij}x^{i}\left(j(1+y)y^{j-1}-2y^j\right)\right)\mathrm{d}x\ \mathrm{d}y\\
&+\int_{\beta(h)}^{\alpha(h)}\frac{1}{(1+y)^2}\sum_{j=0}^na^+_{0j}y^{j}\mathrm{d}y\\
=&\int_{\beta(h)}^{\alpha(h)}\frac{1}{(1+y)^2}\sum_{2i+j\leq n}a^+_{2i,j}(h+hy-y^2)^iy^j\mathrm{d}y\\
&+\int_{\beta(h)}^{\alpha(h)}\frac{1}{(1+y)^2}\sum_{2i+1+j\leq n}a^+_{2i+1,j}(h+hy-y^2)^iy^j\sqrt{h+hy-y^2}\mathrm{d}y\\
&+\int_{\beta(h)}^{\alpha(h)}\frac{1}{(1+y)^3}\sum_{2i+j\leq n}\frac{b^+_{2i,j}}{2i+1}(h+hy-y^2)^i(jy^{j-1}+(j-2)y^j)\sqrt{h+hy-y^2}\mathrm{d}y\\
&+\int_{\beta(h)}^{\alpha(h)}\frac{1}{(1+y)^3}\sum_{2i+j+1\leq n}\frac{b^+_{2i+1,j}}{2i+2}(h+hy-y^2)^{i+1}(jy^{j-1}+(j-2)y^j)\mathrm{d}y\\
=&\sum_{k=0}^{n}\widetilde{m}_{ak}(h)\int_{\beta(h)}^{\alpha(h)}(1+y)^{k-2}\mathrm{d}y+
\sum_{k=0}^{n-1}\overline{m}_{ak}(h)\int_{\beta(h)}^{\alpha(h)}(1+y)^{k-2}\sqrt{h+hy-y^2}\mathrm{d}y\\
&+
\sum_{k=0}^n\overline{m}_{bk}(h)\int_{\beta(h)}^{\alpha(h)}(1+y)^{k-3}\sqrt{h+hy-y^2}\mathrm{d}y
+\sum_{k=0}^{n+1}\widetilde{m}_{bk}(h)\int_{\beta(h)}^{\alpha(h)}(1+y)^{k-3}\mathrm{d}y\\
=&\sum_{k=0}^n\overline{m}_k(h)\int_{\beta(h)}^{\alpha(h)}(1+y)^{k-3}\sqrt{h+hy-y^2}\mathrm{d}y+
\sum_{k=0}^{n+1}\widetilde{m}_k(h)\int_{\beta(h)}^{\alpha(h)}(1+y)^{k-3}\mathrm{d}y,
\end{split}
\end{equation*}
where $\overline{m}_{ik}$, $\widetilde{m}_{ik}$, $i=a,b$, and  $\overline{m}_{k}$, $\widetilde{m}_{k}$ are polynomials of $h$ with degree
\begin{equation}\begin{split}\label{mk1}
&\deg\overline{m}_{ak}\leq\min\{k,n-1-k\},\quad \deg\widetilde{m}_{ak}\leq\min\{k,n-k\},\\
&\deg\overline{m}_{bk}\leq\min\{k,n-k\},\quad \deg\widetilde{m}_{bk}\leq\min\{k,n+1-k\},\\
&\deg\overline{m}_k\leq\min\{k,n-k\},\quad \deg\widetilde{m}_k\leq\min\{k,n+1-k\},
\end{split}\end{equation}
which are determined by Newton's formula and some qualitative analysis, see \cite{LLLZ} for details.
It is worth noting that when $n=0$, $\widetilde{m}_{bk}(h)=0$, for $k=0,1$, and thus $\widetilde{m}_0(h)=0$.

Let
\[
I_k(h)=\int_{\beta(h)}^{\alpha(h)}(1+y)^{k}\sqrt{h+hy-y^2}\mathrm{d}y,\quad
J_k(h)=\int_{\beta(h)}^{\alpha(h)}(1+y)^{k}\mathrm{d}y.
\]
Then
\begin{equation}\label{M1+}
M(h)=\sum_{k=0}^n\overline{m}_k(h)I_{k-3}(h)+
\sum_{k=0}^{n+1}\widetilde{m}_k(h)J_{k-3}(h).
\end{equation}

Next, we need to compute $I_k(h)$ and $J_k(h)$, respectively. For $k>0$, let $u=1+y$, then
\begin{equation*}\begin{split}
I_k(h)=&\int_{u_1}^{u_2}u^k\sqrt{-1+(2+h)u-u^2}\mathrm{d}u\\
=&\int_{u_1}^{u_2}u^k\sqrt{(u-u_1)(u_2-u)}\mathrm{d}u,
\end{split}\end{equation*}
where $u_1=\beta(h)+1$ and $u_2=\alpha(h)+1$ are the two roots of $-1+(2+h)u-u^2=0$.

Two different transformations
$\sqrt{(u-u_1)(u_2-u)}=t(u-u_1)$ and $\sqrt{(u-u_1)(u_2-u)}=t(u_2-u)$ lead to
\begin{equation*}
I_k(h)=2(u_1-u_2)^2\int_{0}^{\infty}\frac{t^2(u_2+u_1t^2)^k}{(1+t^2)^{k+3}}\mathrm{d}t
=2(u_1-u_2)^2\int_{0}^{\infty}\frac{t^2(u_1+u_2t^2)^k}{(1+t^2)^{k+3}}\mathrm{d}t.
\end{equation*}
It follows that
\begin{equation}\begin{split}\label{Ik1}
I_{k}(h)&=\displaystyle\ (u_{1}-u_{2})^{2}\int_{0}^{\infty}\dfrac{t^{2}[(u_{2}+u_{1}t^{2})^{k}+(u_{1}+u_{2}t^{2})^{k}]}{(1+t^{2})^{k+3}}\mathrm{d}t\\
&=\displaystyle\
(u_{1}-u_{2})^{2}\sum_{i+2j=k}r_{ij}(u_{1}+u_{2})^{i}(u_{1}u_{2})^{j}\\
&=\displaystyle h(4+h)
\sum_{i+2j=k}r_{ij}\left(2+h\right)^{i}\\
&=h(4+h)\sum_{i=0}^{k}C_{i,k}h^{i},\quad \quad \quad \quad  k>0,
\end{split}\end{equation}
where $r_{ij}, i+2j=k$, and $C_{i,k}, i=0,1,\cdots,k$ are constants. Moreover
\[
C_{k,k}=\int_{0}^{\infty}\frac{t^{2}(1+t^{2k})}{(1+t^2)^{k+3}}\mathrm{d}t>0.
\]
Direct computations show that
\begin{equation}\begin{split}\label{I1}
&I_{-3}(h)=I_{0}(h)=\frac{h(4+h)\pi}{8},\\
&I_{-2}(h)=I_{-1}(h)=\frac{h\pi}{2}.
\end{split}\end{equation}
We get $J_k(h)$ by direct computations,
\begin{equation}\begin{split}\label{Jk1}
J_k(h)&=\frac{(1+\alpha(h))^{k+1}-(1+\beta(h))^{k+1}}{k+1}\\
&=\frac{(2+h+\sqrt{h(4+h)})^{k+1}-(2+h-\sqrt{h(4+h)})^{k+1}}{(k+1)2^{k+1}}\\
&=\frac{\sqrt{h(4+h)}}{(k+1)2^{k}}\sum_{j=0}^{[k/2]}C_{k+1}^{2j+1}(2+h)^{k-2j}h^j(4+h)^j\\
&=\sqrt{h(4+h)}\sum_{j=0}^{k}s_{j,k}h^j, \quad k>0,
\end{split}\end{equation}
where $s_{i,k}, j=0,1,\cdots,k$ are constants,
\[
s_{k,k}=\frac{1}{(k+1)2^{k}}\sum_{j=0}^{[k/2]}C_{k+1}^{2j+1}=\frac{1}{k+1}>0,
\]
and
\begin{equation}\begin{split}\label{J1}
&J_{-3}(h)=\dfrac{1}{2}\sqrt{h(4+h)}(2+h),\\
&J_{-2}(h)=J_0(h)=\sqrt{h(4+h)},\\
&J_{-1}(h)=2\ln(1+\alpha(h)).
\end{split}\end{equation}
Using \eqref{mk1}-\eqref{J1}, we have the following result.
\begin{proposition}\label{PS11} The first order Melnikov function $M(h)$ for system $S_1$ is:
\begin{equation}\label{M13}
\begin{split}
M(h)
=&\alpha_0h(4+h)+\beta_0\sqrt{h(4+h)},\ n=0,\\
M(h)
=&h(\alpha_{0}+\alpha_{1}h)+\sqrt{h(4+h)}(\beta_0+\beta_1h)+\gamma_0\ln(1+\alpha(h)), \ n=1,\\
M(h)
=&h(\alpha_{0}+\alpha_{1}h)+\sqrt{h(4+h)}(\beta_0+\beta_1h)+(\gamma_0+\gamma_1h)\ln(1+\alpha(h)), \ n=2,\\
M(h)
=&h(\alpha_{0}+\alpha_{1}h)+\sqrt{h(4+h)}(\beta_0+\beta_1h)+(\gamma_0+\gamma_1h+\gamma_2h^2)\ln(1+\alpha(h)),\ n=3,\\
M(h)
=&h\sum_{k=0}^{n-2}\alpha_kh^k+\sqrt{h(4+h)}\sum_{k=0}^{n-2}\beta_kh^k+(\gamma_0+\gamma_1h+\gamma_2h^2)\ln(1+\alpha(h)),\  n\geq4,
\end{split}
\end{equation}
where $\alpha(h)=(h+\sqrt{h(4+h)})/2$, and $\alpha_k, \beta_k, k=0,1,\cdots,n-2$, and $\gamma_0, \gamma_1, \gamma_2$ are constants.
\end{proposition}
\begin{proof}
For $n\leq3$, the computation of $M(h)$ is straightforward. For $n\geq4$,
\begin{equation*}
\begin{split}
\deg\left(\dfrac{\overline{m}_k(h)I_{k-3}(h)}{h}\right)&=\deg\overline{m}_k(h)+\deg \left(\dfrac{I_{k-3}(h)}{h}\right)\\
&\leq\min\{k,n-k\}+k-3+1\\
&\leq n-2,\quad\quad \mbox{for}\ k\geq4,
\end{split}
\end{equation*}
\begin{equation*}
\begin{split}
\deg\left(\frac{\widetilde{m}_k(h)J_{k-3}(h)}{\sqrt{h(4+h)}}\right)&=\deg\widetilde{m}_k(h)+\deg \left(\frac{J_{k-3}(h)}{\sqrt{h(4+h)}}\right)\\
&\leq\min\{k,n+1-k\}+k-3\\
&\leq n-2,\quad\quad \mbox{for}\ k\geq4,\\
\end{split}
\end{equation*}
therefore,
\begin{equation*}\label{M1n}
\begin{split}
M(h)=&\sum_{k=0}^n\overline{m}_k(h)I_{k-3}(h)+
\sum_{k=0}^{n+1}\widetilde{m}_k(h)J_{k-3}(h)\\
=&h(\alpha_{0}+\alpha_{1}h)+\sum_{k=4}^n\overline{m}_k(h)I_{k-3}(h)\\
&+\sqrt{h(4+h)}(\beta_0+\beta_1h)+(\gamma_0+\gamma_1h+\gamma_2h^2)\ln(1+\alpha(h))+\sum_{k=4}^{n+1}\widetilde{m}_k(h)J_{k-3}(h)\\
=&h\sum_{k=0}^{n-2}\alpha_kh^k+\sqrt{h(4+h)}\sum_{k=0}^{n-2}\beta_kh^k+(\gamma_0+\gamma_1h+\gamma_2h^2)\ln(1+\alpha(h)).
\end{split}
\end{equation*}
\end{proof}
\subsection{Independence of the coefficients}\label{sub:S12}

To determine the independence of the coefficients in $M(h)$ obtained in \eqref{M13}, we give the following
lemma. The proof is similar to the computation of $M^{+}(h)$, and thus we omit here.
\begin{lemma}\label{lem:S1}
The following equalities hold.
\begin{equation*}
\begin{split}
&\iint_{D_h^+}\frac{\partial}{\partial x}\left(\frac{x^{2m+1}}{(1+y)^2}\right)\mathrm{d}x\ \mathrm{d}y=\left\{\begin{array}{ll}
h(c_{2m-1}h^{2m-1}+\cdots) & \mbox{if $m>0$,} \\
\frac{\pi}{2}h& \mbox{if $m=0$,}
\end{array} \right.\\
&\iint_{D_h^+}\frac{\partial}{\partial y}\left(\frac{x^{2m}}{(1+y)^2}\right)\mathrm{d}x\ \mathrm{d}y=\left\{\begin{array}{lll}
h(c_{2m-2}h^{2m-2}+\cdots) & \mbox{if $m>1$,} \\
-\frac{\pi}{4}h^2 & \mbox{if $m=1$,}\\
-\frac{\pi}{4}h(h+4) & \mbox{if $m=0$,}
\end{array} \right.\\
&\iint_{D_h^+}\frac{\partial}{\partial y}\left(\frac{x^{2m+1}}{(1+y)^2}\right)\mathrm{d}x\ \mathrm{d}y\\
=&\left\{\begin{array}{ll}
\sqrt{h(4+h)}(d_{2m-1}h^{2m-1}+\cdots)+(-1)^{m}\left(m(h+2)^2+2\right)\ln(1+a(h)) & \mbox{if $m>0$,} \\
\sqrt{h(4+h)}(-\frac{1}{2}h-1)+2\ln(1+a(h))& \mbox{if $m=0$,}
\end{array} \right.\\
&\iint_{D_h^+}\frac{\partial}{\partial x}\left(\frac{x^{2m}}{(1+y)^2}\right)\mathrm{d}x\ \mathrm{d}y\\
=&\left\{\begin{array}{ll}
\sqrt{h(4+h)}(d_{2m-2}h^{2m-2}+\cdots)+2(-1)^{m-1}m(h+2)\ln(1+a(h)) & \mbox{if $m>0$,} \\
0& \mbox{if $m=0$,}
\end{array} \right.\\
&\int_{\beta(h)}^{\alpha(h)}\frac{1}{(1+y)^2}\mathrm{d}y=J_{-2}(h)=\sqrt{h(4+h)},\\
\end{split}
\end{equation*}
\begin{equation*}
\begin{split}
 &\int_{\beta(h)}^{\alpha(h)}\frac{y}{(1+y)^2}\mathrm{d}y=J_{-1}(h)-J_{-2}(h)=2\ln(1+\alpha(h))-\sqrt{h(4+h)},\quad\quad\quad\quad\quad\quad\quad\quad
\end{split}
\end{equation*}
where
\begin{equation*}\begin{split}
&c_{2m-1}=\sum_{k=m}^{2m}C_m^{2m-k}(-1)^{k-m}C_{k-2,k-2}=
\frac{\sqrt{\pi}\ \Gamma(m+\frac{3}{2})}{(2m-1)2^{2m-1}\Gamma(m+1)}, \ m\geq2,
\quad c_1=\frac{3\pi}{8},\\
&c_{2m-2}=-\frac{2}{1+2m}\sum_{k=m}^{2m}C_m^{2m-k}(-1)^{k-m}C_{k-3,k-3}=
-\frac{2 \Gamma(m-\frac{3}{2}) \Gamma(m+\frac{3}{2})}{(2m+1) \Gamma(2m)},\ m\geq2,
\quad c_0=-\frac{\pi}{8},\\
&d_{2m-1}=-\frac{1}{1+m}\sum_{k=m+1}^{2m+2}C_{m+1}^{2m+2-k}(-1)^{k-m-1}\frac{1}{k-2}=
-\frac{\Gamma (m-1) \Gamma (m+2)}{(m+1) \Gamma (2m+1)},\ m\geq2, \quad d_1=\frac{3}{2},\\
&d_{2m-2}=\sum_{k=m}^{2m}C_{m}^{2m-k}(-1)^{k-m}\frac{1}{k-1}=
\frac{m\sqrt{\pi }\ \Gamma (m-1)}{2^{2m-1}\Gamma (m+\frac{1}{2})},\ m\geq2, \quad d_0=-2
\end{split}\end{equation*}
are nonzero constants, and the dots denote the lower-order terms of $h$.
\end{lemma}
\begin{proposition}\label{PS12}
The coefficients in $M(h)$ given in Proposition \ref{PS11} are independent.
\end{proposition}
\begin{proof}
By Lemma \ref{lem:S1}, for $n=0$,
\[
\frac{\partial(\alpha_0,\beta_0)}
{\partial(b^+_{0,0},a^+_{0,0})}
=-\dfrac{\pi}{4}\neq0,
\]
for $n=1$,
\[
\frac{\partial(\alpha_0,\alpha_1,\beta_0,\beta_1,\gamma_0)}
{\partial(a^+_{1,0},b^+_{0,0},a^+_{0,0},b^+_{1,0},a^+_{0,1})}
=\dfrac{\pi^2}{8}\neq0,
\]
for $n=2$,
\[
\frac{\partial(\alpha_0,\alpha_1,\beta_0,\beta_1,\gamma_0,\gamma_1)}
{\partial(a^+_{1,0},b^+_{0,0},a^+_{0,0},b^+_{1,0},a^+_{0,1},a^+_{2,0})}
=\dfrac{\pi^2}{4}\neq0,
\]
for $n\geq3$ odd,
\[
\frac{\partial(\alpha_0,\alpha_1,\alpha_2,\cdots,\alpha_{n-2},\beta_0,\beta_1,\beta_2,\cdots,\beta_{n-2},\gamma_0,\gamma_1,\gamma_2)}
{\partial(a^+_{1,0},a^+_{3,0},b^+_{4,0},\cdots,a^+_{n,0},
a^+_{2,0},b^+_{3,0},a^+_{4,0}\cdots,b^+_{n,0},b^+_{1,0},a^+_{0,0},a^+_{0,1})}
=-\pi c_1\prod_{k=2}^{n-2}c_{k}d_k\neq0,
\]
and for $n\geq4$ even,
\[
\frac{\partial(\alpha_0,\alpha_1,\alpha_2,\cdots,\alpha_{n-2},\beta_0,\beta_1,\beta_2,\cdots,\beta_{n-2},\gamma_0,\gamma_1,\gamma_2)}
{\partial(a^+_{1,0},a^+_{3,0},b^+_{4,0},\cdots,b^+_{n,0},
a^+_{2,0},b^+_{3,0},a^+_{4,0},\cdots,a^+_{n,0},b^+_{1,0},a^+_{0,0},a^+_{0,1})}
=-\pi c_1\prod_{k=2}^{n-2}c_{k}d_k\neq0,
\]
thus the parameters $\alpha_k, \beta_k, k=0,1,\cdots,n-2$, and $\gamma_0, \gamma_1, \gamma_2$ are independent.
\end{proof}
\subsection{Zeros of $M(h)$}
Finally, we estimate the number of zeros of $M(h)$ obtained in Proposition \ref{PS11}.
\begin{proposition}
$H_1(0)=1$, $H_1(n)=n+3$ for $n=1,2,3$, and $H_1(n)=2n$ for $n\geq4$.
\end{proposition}
\begin{proof}
For $n=0$, it is easy to verify that $M(h)$ has at most 1 zero in $(0,+\infty)$, i.e., $H_1(0)\leq1$.

For $n\geq1$, we eliminate the
logarithmic function first by taking derivatives.
\begin{equation*}
M^{(3)}(h)=\dfrac{1}{h^2(4+h)^2\sqrt{h(4+h)}}\sum_{i=0}^{n+1}\widetilde{\beta}_ih^i,\quad \mbox{for} \ n=1,2,3,
\end{equation*}
and
\begin{equation*}
M^{(3)}(h)=\sum_{i=0}^{n-4}\widetilde{\alpha}_ih^i+\dfrac{1}{h^2(4+h)^2\sqrt{h(4+h)}}\sum_{i=0}^{n+1}\widetilde{\beta}_ih^i,
\quad \mbox{for}\ n\geq4.
\end{equation*}
Obviously, $M^{(3)}(h)$ has at most $n+1$ zeros when $n=1,2,3$. Then, it follows from $M(0)=0$ that
\[
H_1(n)\leq n+1+3-1=n+3.
\]

For $n\geq4$, let $F(h)=M^{(3)}(h)$, then it follows from Lemma \ref{lem:P} that
\begin{equation*}
F^{(n-3)}(h)=\dfrac{1}{h^{n-1/2}(4+h)^{n-1/2}}\sum_{i=0}^{n+1}\bar{\beta}_ih^i.
\end{equation*}
Obviously, $F^{(n-3)}(h)$ has at most $n+1$ zeros. Thus, by Rolle's Theorem, $F(h)$, as well as $M^{(3)}(h)$, has at most
$2n-2$ zeros in $(0,+\infty)$. Note that $M(0)=0$, thus
for $n\geq4$,
\[
H_1(n)\leq 2n-2+3-1=2n.
\]

Note that the functions $h^{k+1}, h^k\sqrt{h(4+h)}$, $k=0,1,\cdots, n-2$ and $h^k\ln(1+\alpha(h)), k=0,1,2$ are linearly independent. Hence by Lemma \ref{le:CGP} and Proposition \ref{PS12},
$H_1(0)\geq1$, $H_1(n)\geq n+3$ for $n=1,2,3$ and $H_1(n)\geq 2n$ for $n\geq4$. The Proposition follows.
\end{proof}

\section{Zeros of $M(h)$ for system $S_2$}\label{sec:S2}
Consider the piecewise smooth polynomial perturbations of degree $n$ of system $S_2$:
\begin{equation}
\left(\begin{array}{ll}\dot{x}\\[2ex] \dot{y}\end{array}\right)=\label{S2}\left\{\begin{array}{ll}
\left(\begin{array}{ll}-y+x^2+\frac{\varepsilon }{2}P^+(x,y)\\
x(1+y)+\frac{\varepsilon }{2}Q^+(x,y)\end{array}\right), & \mbox{ $x> 0$,}
\\[2ex]
\left(\begin{array}{ll}-y+x^2+\frac{\varepsilon }{2}P^-(x,y)\\
x(1+y)+\frac{\varepsilon }{2}Q^-(x,y)\end{array}\right), & \mbox{ $x< 0$,}
\end{array} \right.
\end{equation}
where $P^{\pm}(x,y)$ and $Q^{\pm}(x,y)$ are given by \eqref{PQ}. For ${\varepsilon=0}$, the first integral of system \eqref{S2} is
\[
H=\frac{2y+1-x^2}{(1+y)^2},
\]
and the integrating factor is $R=\frac{2}{(1+y)^3}$. Here $L_h^+=\{x\geq0|H=h,h\in(0,1)\}$, $L_h^-=\{x\leq0|H=h,h\in(0,1)\}$, $\mathcal{A}_h=(0,\alpha(h))$ and $\mathcal{B}_h=(0,\beta(h))$, where
\begin{equation}
\alpha(h)=\frac{1-h+\sqrt{1-h}}{h}, \quad \beta(h)=\frac{1-h-\sqrt{1-h}}{h}.
\end{equation}

\subsection{Expression of $M(h)$ and independence of coefficients}
Similarly as the calculus of $M^+(h)$ in subsection \ref{sub:S11}, by \eqref{Mh}, the first order Melnikov function of system \eqref{S2}
is
\begin{equation}\label{M2}
\begin{split}
M(h)=&\iint_{D_h^+}\left[\frac{\partial}{\partial x}\left(\frac{P^+}{(1+y)^3}\right)+\frac{\partial}{\partial y}\left(\frac{Q^+}{(1+y)^3}\right)\right]\mathrm{d}x\ \mathrm{d}y+\int_{b(h)}^{a(h)}\frac{P^+(0,y)}{(1+y)^3}\mathrm{d}y\\
&+\iint_{D_h^-}\left[\frac{\partial}{\partial x}\left(\frac{P^-}{(1+y)^3}\right)+\frac{\partial}{\partial y}\left(\frac{Q^-}{(1+y)^3}\right)\right]\mathrm{d}x\ \mathrm{d}y-\int_{a(h)}^{b(h)}\frac{P^-(0,y)}{(1+y)^3}\mathrm{d}y\\
=&\sum_{k=0}^n\overline{m}_k(h)I_{k-4}(h)+
\sum_{k=0}^{n+1}\widetilde{m}_k(h)J_{k-4}(h),
\end{split}
\end{equation}
where
\[
I_k(h)=\int_{\beta(h)}^{\alpha(h)}(1+y)^{k}\sqrt{1+2y-h(1+y)^2}\mathrm{d}y,\quad
J_k(h)=\int_{\beta(h)}^{\alpha(h)}(1+y)^{k}\mathrm{d}y,
\]
and $\overline{m}_{k}$, $\widetilde{m}_{k}$ are polynomials of $h$ with degree
\begin{equation}\begin{split}\label{mk}
\deg\overline{m}_k\leq[\frac{k}{2}],\quad \deg\widetilde{m}_k\leq[\frac{k}{2}],
\end{split}\end{equation}
which are determined by Newton's formula and some qualitative analysis, see \cite{LLLZ} for details.
It is worth noting that when $n=0$, $\widetilde{m}_0(h)=0$.

We exploit the method than in subsection \ref{sub:S11} to obtain the expressions of $I_k(h)$ and $J_k(h)$ first.
\begin{equation}\begin{split}\label{Ik2}
I_{k}(h)&=\sqrt{h}(1-h)\sum_{i=[(k+1)/2]}^{k}C_{i,k}h^{-i-2},\quad \quad \quad \quad  k>0,
\end{split}\end{equation}
where $C_{i,k}, i=[(k+1)/2],[(k+1)/2]+1,\cdots,k$ are constants, and
\[
C_{k,k}=2^{k+2}\int_{0}^{\infty}\frac{t^{2}(1+t^{2k})}{(1+t^2)^{k+3}}\mathrm{d}t>0,\quad
C_{[(k+1)/2],k}\neq0,
\]
\begin{equation}\begin{split}\label{I2}
&I_{-4}(h)=I_{-3}(h)=\frac{(1-h)\pi}{2},\\
&I_{-2}(h)=(1-\sqrt h)\pi,\\
&I_{-1}(h)=\frac{(1-\sqrt h)\pi}{\sqrt h},\\
&I_0(h)=\frac{(1-h)\pi}{2h^{3/2}}.
\end{split}\end{equation}
\begin{equation}\begin{split}\label{Jk2}
J_k(h)&=\frac{2\sqrt{1-h}}{(k+1)h^{k+1}}\sum_{j=0}^{[k/2]}C_{k+1}^{2j+1}(1-h)^j, \quad k>0,
\end{split}\end{equation}
where $C_{k+1}^{2j+1}, j=0,1,\cdots,[k/2]$ are combinatorial numbers and
\begin{equation}\begin{split}\label{J2}
&J_{-4}(h)=\frac{2}{3}\sqrt{1-h}(4-h),\\
&J_{-3}(h)=J_{-2}(h)=2\sqrt{1-h},\\
&J_{-1}(h)=\ln\left(\frac{1+\sqrt{1-h}}{1-\sqrt{1-h}}\right),\\
&J_0(h)=\frac{2\sqrt{1-h}}{h}.
\end{split}\end{equation}
Thus, using \eqref{M2}-\eqref{J2}, we have
\begin{proposition}\label{PS21} The first order Melnikov function $M(h)$ for system $S_2$ is:
\begin{equation*}
\begin{split}
M(h)=&\alpha_0(1-h)+\beta_0\sqrt{1-h},\quad n=0,\\
M(h)=&\alpha_{0}(1-h)+\sqrt{1-h}(\beta_0+\beta_1h),\quad n=1,\\
M(h)=&\sum_{k=0}^3\alpha_{k}h^{k/2}+\sqrt{1-h}(\beta_0+\beta_1h)
+J_{-1}(h)(\gamma_0+\gamma_1h),\quad n=2,\\
M(h)=&\sum_{k=0}^3\alpha_{k}h^{k/2}+\sum_{k=2-n}^{-1}\alpha_{k}h^{k+1/2}+\sqrt{1-h}\sum_{k=2-n}^1\beta_k h^k+J_{-1}(h)(\gamma_0+\gamma_1h), \quad n\geq3,\\
\end{split}
\end{equation*}
where $J_{-1}(h)$ is given in \eqref{J2}, and $\alpha_k, k=2-n,3-n,\cdots,3$, $\beta_k, k=2-n,3-n,\cdots,1$ and $\gamma_0, \gamma_1$ are constants, and
$\alpha_0,\beta_0$ are independent for $n=0$, $\alpha_0,\beta_0,\beta_1$ are independent for $n=1$, and the coefficients except some $\alpha_k$ in $M(h)$ are independent for $n\geq2$.
\end{proposition}
\begin{proof}
The computation of $M(h)$ is straightforward, and thus is omit here. In the following, we prove the independence of the coefficients.
Since $I_k(1)=0$ and $J_k(1)=0$ for any $k\geq-4$, they mean that
\[
M(1)=\sum_{k=2-n}^3\alpha_k=0,
\]
which implies that $\alpha_k, k=2-n,3-n,\cdots,3$ are linearly dependent, and some $\alpha_k$ can be expressed in others.
More concretely, similarly as that in subsection \ref{sub:S12}, we can obtain that
\[
\frac{\partial(\alpha_0,\beta_0)}
{\partial(b^+_{0,0},a^+_{0,0})}
=-3\pi\neq0,\quad \mbox{for}\ n=0,
\]
\[
\frac{\partial(\alpha_0,\beta_0,\beta_1)}
{\partial(b^+_{0,0},a^+_{0,0},b^+_{1,0})}
=-6\pi\neq0,\quad \mbox{for}\ n=1,
\]
\[
\frac{\partial(\alpha_1,\alpha_2,\alpha_3,\beta_0,\beta_1,\gamma_0,\gamma_1)}
{\partial(b^+_{0,2},b^+_{0,0},b^+_{2,0},a^+_{0,0},b^+_{1,0},a^+_{0,2},a^+_{2,0})}
=6\pi^3\neq0,\quad \mbox{for}\ n=2,
\]
and for $n\geq3$,
\[
\frac{\partial(\alpha_3,\alpha_2,\alpha_1,\alpha_{-1},\alpha_{-2},\cdots,\alpha_{2-n},
\beta_1,\beta_0,\beta_{-1},\cdots,\beta_{2-n},\gamma_0,\gamma_1)}
{\partial(b^+_{2,0},b^+_{0,0},b^+_{0,2},a^+_{1,2},b^+_{0,4},\cdots,b^+_{0,n},
b^+_{1,0},a^+_{0,0},a^+_{0,3}\cdots,a^+_{0,n},a^+_{0,2},a^+_{2,0})}
\neq0.
\]
\end{proof}

\subsection{Zeros of $M(h)$}
Consider the first Melnikov function $M(h)$ obtained in Proposition \ref{PS21}. Then
\begin{proposition}
$H_2(n)=n+1$ for $n=0,1$, and $H_2(n)=2n+2$ for $n\geq2$.
\end{proposition}
\begin{proof}
It is easy to verify that $M(h)$ has at most 1 zero in $(0,1)$ for $n=0$, and at most 2 zeros
in $(0,1)$ for $n=1$, which show that $H_2(n)\leq n+1$ for $n=0,1$. Since the functions $1-h, \sqrt{1-h}$
and $h\sqrt{1-h}$ are linearly independent, by Lemma \ref{le:CGP} and Proposition \ref{PS21}, $H_2(n)\geq n+1$ for
$n=0,1$. The first conclusion of the proposition holds.

For $n\geq2$, we get the second derivative of $M(h)$ first.
\[
M''(h)=\frac{1}{\sqrt{h}}\sum_{k=1-n}^{0}\widetilde\alpha_{k}h^{k}
+\frac{1}{(1-h)^{3/2}}\sum_{k=-n}^{1}\widetilde\beta_{k}h^{k},
\]
where $\widetilde\alpha_{k}, \widetilde\beta_{k}$ are liner combinations of $\alpha_{k}, \beta_{k}$, and are independent.

Let $F(h)=h^{n-\frac{1}{2}}M''(h)$, then
\[
F(h)=\sum_{k=0}^{n-1}\widetilde\alpha_{k-n+1}h^{k}
+\frac{1}{h^{1/2}(1-h)^{3/2}}\sum_{k=0}^{n+1}\widetilde\beta_{k-n}h^{k},
\]
which has the same zeros as $M''(h)$ in $(0,1)$. By Lemma \ref{lem:P},
\[
F^{(n)}(h)=\frac{1}{h^{n+1/2}(1-h)^{n+3/2}}\sum_{k=0}^{n+1}\overline\beta_{k}h^{k}.
\]
Obviously, $F^{(n)}(h)$ has at most $n+1$ zeros in $(0,1)$. Thus, $F(h)$, as well as $M''(h)$, has at most $2n+1$ zeros in $(0,1)$. It follows that $M(h)$ has at
most $2n+3$ zeros. Note that $M(1)=0$, thus $M(h)$ has at
most $2n+2$ zeros in $(0,1)$. That is $H_2(n)\leq 2n+2$ for $n\geq2$.

On the other hand,  $\sum_{k=2-n}^3\alpha_k=0$ shows that
$$\alpha_0=-\sum_{k=2-n}^{-1}\alpha_k-\sum_{k=1}^3\alpha_k,$$
 and $M(h)$ can be rewritten as
\begin{equation}\label{Mr}
M(h)=\sum_{k=1}^3\alpha_{k}(h^{k/2}-1)+\sum_{k=2-n}^{-1}\alpha_{k}(h^{k+1/2}-1)+\sqrt{1-h}\sum_{k=2-n}^1\beta_k h^k+J_{-1}(h)(\gamma_0+\gamma_1h).
\end{equation}
Since $h^{k/2}-1, k=1,2,3$, $h^{k+1/2}-1, k=2-n,\cdots,-1$, $h^k\sqrt{1-h}, k=2-n,\cdots,1$ and $h^kJ_{-1}(h), k=0,1$ are linearly independent,
it follows from Lemma \ref{le:CGP} and Proposition \ref{PS21} that $H_2(n)\geq2n+2$ for $n\geq2$. The second part of the proposition
follows. Specifically, $H_2(2)=6$.
\end{proof}
\section{Zeros of $M(h)$ for system $S_3$}\label{sec:S3}
Consider the piecewise smooth polynomial perturbations of degree $n$ of system $S_3$:
\begin{equation}
\left(\begin{array}{ll}\dot{x}\\[2ex] \dot{y}\end{array}\right)=\label{S3}\left\{\begin{array}{ll}
\left(\begin{array}{ll}-y+\frac{1}{4}x^{2}+\varepsilon P^+(x,y)\\
x(1+y)+\varepsilon Q^+(x,y)\end{array}\right), & \mbox{ $x> 0$,}
\\[2ex]
\left(\begin{array}{ll}-y+\frac{1}{4}x^{2}+\varepsilon P^-(x,y)\\
x(1+y)+\varepsilon Q^-(x,y)\end{array}\right), & \mbox{ $x< 0$,}
\end{array} \right.
\end{equation}
where $P^{\pm}(x,y)$ and $Q^{\pm}(x,y)$ are given by \eqref{PQ}. For ${\varepsilon=0}$, the first integral of system \eqref{S3} is
\[
H=\dfrac{(x^2+4y+8)^2}{1+y},
\]
and the integrating factor is $R=\dfrac{4(x^2+4y+8)}{(1+y)^2}$.
\subsection{Expression of $M(h)$ and independence of the coefficients}
Let $x=2\bar x$ and $y=\bar y(2+\bar y)$, system \eqref{S3} is transformed into the following perturbed system $S_1$
(we omit ``\ $\bar{}$\ '' below for convenience):
\begin{equation}
\left(\begin{array}{ll}\dot{x}\\[2ex] \dot{y}\end{array}\right)=\label{S31}\left\{\begin{array}{ll}
\left(\begin{array}{ll}-y+\frac{1}{2}x^{2}-\frac{1}{2}y^{2}+\frac{\varepsilon}{2} P^+\left(2x,y(2+y)\right)\\
x(1+y)+\frac{\varepsilon}{2(1+y)} Q^+\left(2x,y(2+y)\right)\end{array}\right), & \mbox{ $x> 0$,}
\\[2ex]
\left(\begin{array}{ll}-y+\frac{1}{2}x^{2}-\frac{1}{2}y^{2}+\frac{\varepsilon}{2} P^-\left(2x,y(2+y)\right)\\
x(1+y)+\frac{\varepsilon}{2(1+y)} Q^-\left(2x,y(2+y)\right)\end{array}\right), & \mbox{ $x< 0$.}
\end{array} \right.
\end{equation}
Thus, the first order Melnikov function of system \eqref{S31}
is
\begin{equation*}\label{M}
\begin{split}
M(h)=&
\iint_{D_h^+}\left[\frac{\partial}{\partial x}\left(\frac{P^+}{(1+y)^2}\right)+\frac{\partial}{\partial y}\left(\frac{Q^+}{(1+y)^3}\right)\right]\mathrm{d}x\ \mathrm{d}y+\int_{\beta(h)}^{\alpha(h)}\frac{P^+|_{x=0}}{(1+y)^2}\mathrm{d}y\\
&+\iint_{D_h^-}\left[\frac{\partial}{\partial x}\left(\frac{P^-}{(1+y)^2}\right)+\frac{\partial}{\partial y}\left(\frac{Q^-}{(1+y)^3}\right)\right]\mathrm{d}x\ \mathrm{d}y+\int_{\beta(h)}^{\alpha(h)}\frac{P^-|_{x=0}}{(1+y)^2}\mathrm{d}y,\quad h>0,\\
\end{split}
\end{equation*}
where $P^{\pm}$ and $Q^{\pm}$ can be rewritten as
\begin{equation*}\begin{split}
P^{\pm}=\sum_{i+j=0}^{n}c_{ij}^{\pm}x^i(1+y)^{2j},\quad Q^{\pm}=\sum_{i+j=0}^{n}d_{ij}^{\pm}x^i(1+y)^{2j},
\end{split}\end{equation*}
with $c_{ij}^{\pm}$ and $d_{ij}^{\pm}$ being linear combinations of $a_{ij}^{\pm}$ and $b_{ij}^{\pm}$, respectively.

By the results obtained in section \ref{sec:S1}, we have the following statements.
\begin{proposition}\label{PS31}The first order Melnikov function $M(h)$ for system $S_3$ is:
\begin{equation*}\label{M3}
\begin{split}
M(h)=&\alpha_0h(4+h)(2+h)+\beta_0\sqrt{h(4+h)},\quad\ n=0,\\
M(h)=&h(\alpha_{0}+\alpha_{1}(4+h)(2+h))+\sqrt{h(4+h)}(\beta_0+\beta_1(2+h)^2),\quad\ n=1,\\
M(h)=&h(\alpha_{0}+\alpha_{1}h+\alpha_{2}h^2)+\sqrt{h(4+h)}(\beta_0+\beta_1(2+h)^2)+\gamma_0(2+h)\ln(1+\alpha(h)),\ n=2,\\
M(h)
=&h(\alpha_{-2}+\alpha_{-1}h+\alpha_{0}h^2)+h(4+h)\sum_{i=1}^{n-2}\alpha_i(2+h)^{2i}\\
&+\sqrt{h(4+h)}\sum_{i=0}^{n-1}\beta_i(2+h)^{2i}+\gamma_0(2+h)\ln(1+\alpha(h)),\quad\ n=3,4,\\
M(h)
=&h(\alpha_{-2}+\alpha_{-1}h+\alpha_{0}h^2)+h(4+h)\sum_{i=1}^{n-2}\alpha_i(2+h)^{2i}\\
&+\sqrt{h(4+h)}\sum_{i=0}^{n-1}\beta_i(2+h)^{2i}+(2+h)(\gamma_0+\gamma_1(2+h)^2)\ln(1+\alpha(h)), \quad\ n\geq5,
\end{split}
\end{equation*}
where $\alpha(h)=(h+\sqrt{h(4+h)})/2$. Moreover, all the coefficients of $M(h)$ are independent.
\end{proposition}

\subsection{Zeros of $M(h)$}
Consider the first order Melnikov function $M(h)$ obtained in Proposition \ref{PS31}. Then
\begin{proposition}
$H_3(0)=1$, $H_3(n)=2n+1$ for $n=1,2$, and $H_3(n)=2n+2$ for $n\geq3$.
\end{proposition}
\begin{proof}
For $n=0$, $M(h)$ has at most 1 zero in $(0,+\infty)$ by a direct computation. Notice that the two generating functions are linearly independent, by Lemma \ref{le:CGP} and Proposition \ref{PS31},
$H_3(0)\geq1$, thus $H_3(0)=1$.

For $n=1$, let
\begin{equation*}\begin{split}
&f_1=h, \quad f_2=h(4+h)(2+h),\quad f_3=\sqrt{h(4+h)}, \quad f_4=\sqrt{h(4+h)}(2+h)^2.
\end{split}\end{equation*}
Then
\begin{equation*}\begin{split}
W_1&=h>0,\\
W_2&=2h^2(3+h)>0,\\
W_3&=\frac{4(18+16h+3h^2)\sqrt{h(4+h)}}{(4+h)^2}>0,\\
W_4&=-\frac{192(6+h)}{h(4+h)^3}<0,
\end{split}\end{equation*}
hence $M(h)$ has at most 3 zeros in $(0,+\infty)$ by \emph{Chebyshev} criterion.

For $n=2$, it follows from
\begin{equation}\label{S32}
M^{(4)}(h)=-\frac{2\left(8(3\beta_0+48\beta_1+4\gamma_0)+2(12\beta_0+12\beta_1+13\gamma0)(2+h)^2-\gamma_0 (2+h)^4\right)}{(h(4+h))^{7/2}}
\end{equation}
that $M^{(4)}(h)$ has at most $2$ zeros in $(0,+\infty)$. Thus by Rolle's Theorem,
$M(h)$ has at most $6$ zeros. Note that $M(0)=0$, $M(h)$ has at most $5$ zeros in $(0,+\infty)$.
Since the generating functions of $M(h)$ are linearly independent, $H_3(n)=2n+1$ for $n=1,2$ follows from Lemma \ref{le:CGP} and Proposition \ref{PS31}.

For $n\geq3$,
\[
M^{(4)}(h)=\sum_{k=0}^{n-3}\widetilde{\alpha}_k(2+h)^{2k}+\frac{1}{(h(4+h))^{7/2}}\sum_{k=0}^{n+1}\widetilde{\beta}_k(2+h)^{2k}.
\]
Let $F(h)=M^{(4)}(h)$ and $z=(2+h)^2$, then we have
\[
F(z)=\sum_{k=0}^{n-3}\widetilde{\alpha}_k z^{k}+\frac{1}{(z-4)^{7/2}}\sum_{k=0}^{n+1}\widetilde{\beta}_kz^{k}, \quad z>4.
\]
It is easy to obtain that
\[
F^{(n-2)}(z)=\frac{1}{(z-4)^{n+3/2}}\sum_{k=0}^{n+1}\bar{\beta}_kz^{k}, \quad z>4.
\]
Hence, by Rolle's Theorem $F(z)$ have at most $2n-1$ zeros in $(4,+\infty)$, which implies that
$M^{(4)}(h)$ has at most $2n-1$ zeros in $(0,+\infty)$. Thus
$M(h)$ has at most $2n+3$ zeros. It follows from $M(0)=0$ that $M(h)$ has at most $2n+2$ zeros in $(0,+\infty)$.

We remark that the upper bound is sharp when $n\geq5$ by Lemma \ref{le:CGP} and Proposition \ref{PS31}, i.e., $H_3(n)=2n+2$ for $n\geq5$. In what follows, we show that the upper
bound also can be reached when $n=3,4$. Here \emph{Chebyshev} criterion does not work since the
last Wronskian determinant has zeros. Lemma \ref{le:NT} is needed for $n=3$, in which case the
last Wronskian determinant has a simple zero. However, the
last Wronskian determinant has two simple zeros for $n=4$, we apply Lemma \ref{le:NT2} to show
the upper bound can be achieved.

For $n=3$, let
\begin{equation*}\begin{split}
&f_i=h^{i}, \ i=1,2,3,\\
&f_4=h(4+h)(2+h)^2,\\
&f_i=\sqrt{h(4+h)}(2+h)^{2(i-5)},\ i=5,6,7,\\
&f_8=(2+h)\ln(1+\alpha(h)),
\end{split}\end{equation*}
then
\begin{equation*}\begin{split}
W_1=&h>0,\\
W_2=&h^2>0,\\
W_3=&2h^3>0,\\
W_4=&12h^4>0,\\
W_5=&\frac{576(35+30h+9h^2+h^3)\sqrt{h(4+h)}}{(4+h)^4}>0,\\
W_6=&-\frac{138240(2+h)(70 + 56 h + 14 h^2 + h^3)}{h^3(4+h)^7}<0,\\
W_7=&-\frac{99532800(2+h)^3}
{h^6(4+h)^{10}\sqrt{h(4+h)}}Y_7<0,\\
W_8=&\frac{4777574400 (2 + h)^6((1 - 2 h)^2 + 11 h^2 + 8 h^3 + h^4)}
{h^{13}(4+h)^{13}\sqrt{h(4+h)}}Y_8,
\end{split}\end{equation*}
where
\begin{equation*}\begin{split}
Y_7=&112(1-h)^2+560h^2+1192h^3+676h^4+178h^5+22h^6+h^7>0,\\
Y_8=&1680\ln(1+\alpha(h))\\
&-h\frac{1057 + 30452 h^2 + 47 (20 h - 7)^2+ 31164 h^3 + 5334 h^4 +
 336 h^5 + 88 h^6 + 18 h^7 + h^8}{\sqrt{h(4+h)}((1 - 2 h)^2 + 11 h^2 + 8 h^3 + h^4)}.
\end{split}\end{equation*}
Since
\[
Y'_8=-\frac{2h^3(2 + h)(-1 + 8 h + 2 h^2)Y_7}{(4 + h)(1 - 4 h + 15 h^2 + 8 h^3 + h^4)^2\sqrt{h(4+h)}},
\]
which has a zero at $h^*=(3\sqrt{2}-4)/2$, we obtain that $Y_8$ increases when $h\in(0,h^*)$ and
decreases when $h\in(h^*,+\infty)$. Note that
$\lim_{h\rightarrow0^{+}}Y_8=0$ and $\lim_{h\rightarrow +\infty}Y_8=-\infty$.
Thus, $Y_8$ has a simple zero in $h\in(0,+\infty)$, equivalently, $W_8$ has a simple
zero in $h\in(0,+\infty)$. It follows from  Lemma \ref{le:NT} that $M(h)$ can have $8$ zeros in $h\in(0,+\infty)$.

For $n=4$, let
\begin{equation*}\begin{split}
f_i&=h^{i},\ i=1,2,3 \\
f_i&=h(4+h)(2+h)^{2(i-3)},\ i=4,5,\\
f_i&=\sqrt{h(4+h)}(2+h)^{2(i-6)},\ i=6,7,8,9,\\
f_{10}&=(2+h)\ln(1+\alpha(h)),
\end{split}\end{equation*}
then all of $W_i, i=1,2,3,4$ are the same than in $n=3$ and are positive for $h>0$, and
\begin{equation*}\begin{split}
W_5=&288 h^5 (12 + 5 h)>0,\\
W_6=&-\frac{69120 (2 + h) (756 + 847 h + 364 h^2 + 73 h^3 + 6 h^4)\sqrt{h(4+h)}}{(4+h)^5}<0,\\
W_7=&-\frac{696729600 (2 + h)^3 (3 + h) (84 + 63 h + 15 h^2 + h^3)}{h^4 (4 + h)^9}<0,\\
W_8=&\frac{501645312000(2 + h)^6}
{h^8(4+h)^{13}\sqrt{h(4+h)}}\big(1038 + 588 (h - 1)^2 + 3384 h^3 + 390 (h^2 - 1)^2 + 2794 h^4\\& + 1268 h^5
   + 258 h^6 + 26 h^7 + h^8\big)>0,\\
W_9=&\frac{20226338979840000 (2 + h)^{10} }
{h^{13} (4 + h)^{18}}Y_9>0,\\
W_{10}=&\frac{5825185626193920000 (2 + h)^{15}Z_1}
{h^{22} (4 + h)^{22}}
\left(-5040 \ln(1+\alpha(h))
+\frac{hZ_2}{\sqrt{h(4+h)}Z_1}\right),
\end{split}\end{equation*}
where
\begin{equation*}\begin{split}
Y_9=&700 + 25424 h^2 + 257 (12 h - 2)^2 + 47228 h^3 + 15520 h^2 (h^2 - 1)^2+27588 h^6 \\
& + 1348 h^3 (h^2 - 1)^2 + 36222 h^7+ 15253 h^8 + 3520 h^9 + 472 h^{10} +
 34 h^{11} + h^{12}>0,\\
Z_1=&15 - 160 h + 1304 h^2 - 1248 h^3 - 76 h^4 + 920 h^5 + 450 h^6 +
 80 h^7 + 5 h^8>0,\\
Z_2=&151200 - 1600200 h + 13008660 h^2 - 11470860 h^3 - 1925118 h^4 +
 9318900 h^5 + 5293110 h^6\\
 & + 1120770 h^7 + 102665 h^8 + 6340 h^9 +
 1110 h^{10} + 130 h^{11} + 5 h^{12}>0,
\end{split}\end{equation*}
by Sturm Theorem. Similarly, denote the function in the parenthesis of $W_{10}$ by
$Y_{10}$, then
\[
Y'_{10}=\frac{20h^4(2 + h) (3 - 60 h + 65 h^2 + 40 h^3 + 5 h^4)Y_9 }{(4 + h)Z_1^2\sqrt{h(4+h)}}
\]
has two simple zeros $h_1^*$ and $h_2^*$ in $h\in(0,+\infty)$ with $h_1^*<h_2^*$. Therefore $Y_{10}$ increases when $h\in(0,h_1^*)\bigcup(h_2^*,+\infty)$ and
decreases when $h\in(h_1^*,h_2^*)$. Notice that
$\lim_{h\rightarrow0^{+}}Y_{10}=0$, $\lim_{h\rightarrow 1/2}Y_{10}\approx-0.58<0$, and $\lim_{h\rightarrow +\infty}Y_{10}=+\infty$. Hence $Y_{10}$ has two simple zeros
$h_1$ and $h_2$. Obviously,
the ordered set $(f_1,f_2,\cdots,f_{10})$ satisfies $W_9(h_1)\neq0$, $W_{10}(h_1)=0$ and $W'_{10}(h_1)\neq0$. It follows from Lemma \ref{le:NT2} that
$M(h)$ can have $10$ zeros in $(0,+\infty)$.
\end{proof}
\begin{remark}

If $a_{00}^{\pm}=b_{00}^{\pm}=0$, then $\beta_0=-4\beta_1-\gamma_0$ (see Appendix A.2 for the specific expressions of $\beta_0, \beta_1$ and  $\gamma_0$), and
from \eqref{S32},
\[
M^{(4)}(h)=\frac{2(72\beta_1+2\gamma_0+\gamma_0(2+h)^2)}{(h(4+h))^{5/2}}.
\]
Thus, $M(h)$ has at most $4$ zeros in $(0,+\infty)$. This result is consistent with that in \cite{LC}.
In fact, by the change
\[
h=2\left(-1+\frac{1}{\sqrt{1-r^2}}\right),
\]
$M(h)$ can be translated into
\[
M(r)=-\frac{F(r)}{(1-r^2)^{3/2}},
\]
where $F(r)$ is the averaged function in \cite{LC}, and we omit the difference in the coefficients. This shows that
the first order Melnikov function and the first order Averaged function may be equivalent in investigating the number of limit cycles of piecewise smooth polynomial differential systems.
\end{remark}

\section{Zeros of $M(h)$ for system $S_4$}\label{sec:S4}
Consider the piecewise smooth polynomial perturbations of degree $n$ of system $S_4$:
\begin{equation}
\left(\begin{array}{ll}\dot{x}\\[2ex] \dot{y}\end{array}\right)=\label{S4}\left\{\begin{array}{ll}
\left(\begin{array}{ll}-y+2x^{2}-\frac{1}{2}y^{2}+\frac{\varepsilon }{8}P^+(x,y)\\
x(1+y)+\frac{\varepsilon }{8}Q^+(x,y)\end{array}\right), & \mbox{ $x> 0$,}
\\[2ex]
\left(\begin{array}{ll}-y+2x^{2}-\frac{1}{2}y^{2}+\frac{\varepsilon }{8}P^-(x,y)\\
x(1+y)+\frac{\varepsilon }{8}Q^-(x,y)\end{array}\right), & \mbox{ $x< 0$,}
\end{array} \right.
\end{equation}
where $P^{\pm}(x,y)$ and $Q^{\pm}(x,y)$ are given by \eqref{PQ}. For ${\varepsilon=0}$, the first integral of system \eqref{S4} is
\[
H=\frac{4x^2-2(y+1)^2+1}{(1+y)^4},
\]
and the integrating factor is $R=\frac{8}{(1+y)^5}$. Here $L_h^+=\{x\geq0|H=h,-1<h<0\}$ and $L_h^-=\{x\leq0|H=h,-1<h<0\}$ are the right part and the left part of the periodic orbits surrounding the origin. $\mathcal{A}_h=(0,\alpha(h))$ and $\mathcal{B}_h=(0,\beta(h))$, where
\begin{equation}
\alpha(h)=-1+\sqrt{\frac{1+\sqrt{1+h}}{-h}}, \quad \beta(h)=-1+\sqrt{\frac{1-\sqrt{1+h}}{-h}}.
\end{equation}
\subsection{Expression of $M(h)$}
By \eqref{Mh}, the first order Melnikov function of system \eqref{S4}
is
\begin{equation}\label{M4}
\begin{split}
M(h)=&\iint_{D_h^+}\left[\frac{\partial}{\partial x}\left(\frac{P^+}{(1+y)^5}\right)+\frac{\partial}{\partial y}\left(\frac{Q^+}{(1+y)^5}\right)\right]\mathrm{d}x\ \mathrm{d}y+\int_{\beta(h)}^{\alpha(h)}\frac{P^+(0,y)}{(1+y)^5}\mathrm{d}y\\
&+\iint_{D_h^-}\left[\frac{\partial}{\partial x}\left(\frac{P^-}{(1+y)^5}\right)+\frac{\partial}{\partial y}\left(\frac{Q^-}{(1+y)^5}\right)\right]\mathrm{d}x\ \mathrm{d}y-\int_{\beta(h)}^{\alpha(h)}\frac{P^-(0,y)}{(1+y)^5}\mathrm{d}y\\
=&\sum_{k=0}^{n+2[\frac{n}{2}]}\overline{m}_k(h)I_{k-6}(h)+
\sum_{k=0}^{n+3+2[\frac{n-1}{2}]}\widetilde{m}_k(h)J_{k-6}(h),
\end{split}
\end{equation}
where
\[
I_k(h)=\int_{\beta(h)}^{\alpha(h)}(1+y)^{k}\sqrt{-1+2(1+y)^2+h(1+y)^4}\mathrm{d}y,\quad
J_k(h)=\int_{\beta(h)}^{\alpha(h)}(1+y)^{k}\mathrm{d}y,
\]
and $\overline{m}_{k}$, $\widetilde{m}_{k}$ are polynomials of $h$ with degree
\begin{equation}\begin{split}\label{mk4}
&\deg\overline{m}_k\leq[\frac{k}{4}],\quad \ \deg\widetilde{m}_k\leq[\frac{k}{4}],
\end{split}\end{equation}
which are determined by Newton's formula and some qualitative analysis, see \cite{LLLZ} for details.
It is worth noting that  $\widetilde{m}_0(h)=0$ for $n=0$, $\overline{m}_{2n-1}=0$ for $n$ even, and
$\widetilde{m}_{2n+1}=0$ for $n$ odd,  and
\begin{equation}\label{mkn}
\begin{split}
\overline{m}_k(h)&=\left\{\begin{array}{ll}
\displaystyle\sum_{i=0}^{[k/4]}c_ih^i & \mbox{for $k\leq n$,} \\
\displaystyle\sum_{i=[(k+1-n)/2]}^{[k/4]}c_ih^i & \mbox{for $k\geq n+1$,}
\end{array} \right.\\
\widetilde{m}_k(h)&=\left\{\begin{array}{ll}
\displaystyle\sum_{i=0}^{[k/4]}c_ih^i \quad & \mbox{for $k\leq n+1$,} \\
\displaystyle\sum_{i=[(k-n)/2]}^{[k/4]}c_ih^i \quad & \mbox{for $k\geq n+2$.}
\end{array} \right.
\end{split}
\end{equation}
Formula \eqref{mkn} is essential to obtain a more precise expression of $M(h)$.

Using the results in \cite{LLLZ},  the expressions of $I_k(h)$ are as follows.

For $k$ odd, i.e. $k=2m+1$,
\begin{equation}\begin{split}\label{Ik4}
I_{2m+1}(h)
=4(-h)^{-\frac{3}{2}}(1+h)\sum_{i=0}^{[m/2]}C_{i,m}(-h)^{i-m}, \quad  m\geq0,
\end{split}\end{equation}
and
\begin{equation}\begin{split}\label{I4}
&I_{-5}(h)=\frac{(1+h)\pi}{4},\\
&I_{-3}(h)=\frac{1}{2}(1-\sqrt{-h})\pi,\\
&I_{-1}(h)=\frac{1}{2}((-h)^{-\frac{1}{2}}-1)\pi.
\end{split}\end{equation}
For $k$ even, i.e. $k=2m$,
\begin{equation}\begin{split}\label{Ik41}
I_{2m}(h)=&(-h)^{-m-1}(f_m \bar I_{2}+g_m \bar I_0), \quad m\geq1,
\end{split}\end{equation}
and
\begin{equation}\begin{split}\label{I41}
&I_{-6}(h)=\bar I_2,\\[1ex]
&I_{-4}(h)=\bar I_0,\\
&I_{-2}(h)=\dfrac{4\bar I_0-5\bar I_2}{h},\\
&I_{0}(h)=-\dfrac{\bar I_0}{h},
\end{split}\end{equation}
where $f_m$ and $g_m$ are polynomials with respect to $h$ with $\deg f_m=[(m-1)/2]$ and
$\deg g_m=[m/2]$, and
\begin{equation}\label{barIk}
\bar I_k=\int_{u_1}^{u_2}u^k\sqrt{h+2u^2-u^4}\mathrm{d}u.
\end{equation}
Here $u_1=\sqrt{-h}(1+\beta(h))$ and $u_2=\sqrt{-h}(1+\alpha(h))$ are the roots of $h+2u^2-u^4=0$,
and $\bar I_0$ and $\bar I_2$ satisfy the Picard-Fuchs equation:
\begin{equation}\label{I02}
4h(1+h)\left(\begin{array}{ll}\bar I_0'\\[1ex] \bar I_2'\end{array}\right)=
\left(\begin{array}{ll}4+3h &-5\\[1ex] -h &5h\end{array}\right)
\left(\begin{array}{ll}\bar I_0\\[1ex] \bar I_2\end{array}\right).
\end{equation}

We get $J_k(h)$ by direct computations,
\begin{equation*}\begin{split}
J_k(h)&=\frac{(1+\alpha(h))^{k+1}-(1+\beta(h))^{k+1}}{k+1}\\
&=\frac{\left(1+\sqrt{1+h}\right)^{\frac{k+1}{2}}-\left(1-\sqrt{1+h}\right)^{\frac{k+1}{2}}}{(k+1)(-h)^{\frac{k+1}{2}}},\quad k\neq-1.
\end{split}\end{equation*}
If $k$ is odd, i.e., $k=2m+1$, then
\begin{equation}\begin{split}\label{Jk4}
J_{2m+1}(h)=&\frac{\left(1+\sqrt{1+h}\right)^{m+1}-\left(1-\sqrt{1+h}\right)^{m+1}}{(2m+2)(-h)^{m+1}}\\
=&\frac{\sqrt{1+h}}{(m+1)(-h)^{m+1}}\sum_{j=0}^{[\frac{m}{2}]}C_{m+1}^{2j+1}(1+h)^{j},\quad \quad m\geq0,
\end{split}\end{equation}
and if $k$ is even, i.e., $k=2m$, then
\begin{equation}\begin{split}\label{Jk5}
J_{2m}(h)=&\frac{\left(1+\sqrt{1+h}\right)^{m+\frac{1}{2}}-\left(1-\sqrt{1+h}\right)^{m+\frac{1}{2}}}{(2m+1)(-h)^{m+\frac{1}{2}}}\\
=&\frac{\left[\left(1+\sqrt{1+h}\right)^{m+\frac{1}{2}}-\left(1-\sqrt{1+h}\right)^{m+\frac{1}{2}}\right]
\left[\left(1+\sqrt{1+h}\right)^{\frac{1}{2}}+\left(1-\sqrt{1+h}\right)^{\frac{1}{2}}\right]}
{(2m+1)(-h)^{m+\frac{1}{2}}\left[\left(1+\sqrt{1+h}\right)^{\frac{1}{2}}+\left(1-\sqrt{1+h}\right)^{\frac{1}{2}}\right]}\\
=&\frac{\left(1+\sqrt{1+h}\right)^{m+1}-\left(1-\sqrt{1+h}\right)^{m+1}+
\sqrt{-h}\left[\left(1+\sqrt{1+h}\right)^{m}-\left(1-\sqrt{1+h}\right)^{m}\right]}
{(2m+1)(-h)^{m+\frac{1}{2}}\sqrt{2(1+\sqrt{-h})}}
\\
=&\frac{\sqrt{2(1-\sqrt{-h})}}
{(2m+1)(-h)^{m+\frac{1}{2}}}\left(\displaystyle\sum_{j=0}^{[\frac{m}{2}]}C_{m+1}^{2j+1}(1+h)^{j}
+\sqrt{-h}\displaystyle\sum_{j=0}^{[\frac{m-1}{2}]}C_{m}^{2j+1}(1+h)^{j}\right)\\
=&\frac{\sqrt{2(1-\sqrt{-h})}}
{(2m+1)(-h)^{m+\frac{1}{2}}}\displaystyle\sum_{j=0}^{m}C_{j,m}\left(\sqrt{-h}\right)^j, \quad \quad m>0,
\end{split}\end{equation}
where $C_{j,m}, j=0,1,\cdots,m$ are constants, and
\begin{equation}\begin{split}\label{J4}
&J_{-6}(h)=\dfrac{\sqrt{2(1-\sqrt{-h})}}{5}(4+2\sqrt{-h}+h),\\
&J_{-5}(h)=J_{-3}(h)=\sqrt{1 + h},\\
&J_{-4}(h)=\dfrac{\sqrt{2(1-\sqrt{-h})}}{3}(2+\sqrt{-h}),\\
&J_{-2}(h)=\sqrt{2(1-\sqrt{-h})},\\
&J_{-1}(h)=\ln\frac{1+\sqrt{1+h}}{\sqrt{-h}},\\
&J_{0}(h)=\dfrac{\sqrt{2(1-\sqrt{-h})}}{\sqrt{-h}}.
\end{split}\end{equation}
Using \eqref{M4}-\eqref{I41} and \eqref{Jk4}-\eqref{J4}, we have
\begin{proposition} The first order Melnikov function $M(h)$ for system $S_4$ is:
\begin{equation*}
\begin{split}
M(h)=&\alpha_0\bar I_2+\beta_0\sqrt{1 + h},\quad\ n=0,\\
M(h)=&\alpha_0\bar I_2+\delta_0(1+h)+\sqrt{1-\sqrt{-h}}(\beta_0(2+\sqrt{-h})+\beta_1h)
+\gamma_0\sqrt{1 + h},\quad\ n=1,\\
M(h)=&\alpha_0\bar I_2+\xi_0\bar I_0+\delta_0(1+h)+\sqrt{1-\sqrt{-h}}(\beta_0(2+\sqrt{-h})+\beta_1h)
+\gamma_0\sqrt{1 + h}\\
&+\eta_0hJ_{-1}(h),\quad\ n=2,\\
M(h)
=&\alpha_0\bar I_2+\xi_0\bar I_0+(1-\sqrt{-h})(\delta_0+\delta_1\sqrt{-h})+\sqrt{1-\sqrt{-h}}P_2(\sqrt{-h})
+\gamma_0\sqrt{1 + h}\\
&
+\eta_0hJ_{-1}(h),\quad\ n=3,\\
M(h)
=&\alpha_{-1}h^{-1}(4\bar I_0-5\bar I_2)+\alpha_0\bar I_2+\xi_0\bar I_0+(1-\sqrt{-h})(\delta_0+\delta_1\sqrt{-h})\\&+\sqrt{1-\sqrt{-h}}P_2(\sqrt{-h})
+\gamma_0\sqrt{1 + h}
+(\eta_0+\eta_1h)J_{-1}(h),\quad\ n=4,\\
M(h)
=&\alpha_{-1}h^{-1}(4\bar I_0-5\bar I_2)+\alpha_0\bar I_2+\xi_0\bar I_0+((-h)^{-\frac{1}{2}}-1)P_2(\sqrt{-h})\\&+\sqrt{1-\sqrt{-h}}(-h)^{-\frac{1}{2}}
P_3(\sqrt{-h})
+\gamma_0\sqrt{1 + h}
+(\eta_0+\eta_1h)J_{-1}(h),\quad\ n=5,
\end{split}
\end{equation*}
and for $n\geq6$ even,
\begin{equation*}
\begin{split}
M(h)=&((-h)^{-\frac{1}{2}}-1)P_2(\sqrt{-h})
+(-h)^{\frac{5-n}{2}}(1+h)P_{\frac{n-6}{2}}(h)+(-h)^{\frac{4-n}{2}}\left(P_{\frac{n-4}{2}}(h)\bar I_2+\bar P_{\frac{n-4}{2}}(h)\bar I_0\right)\\[1ex]
&+(\eta_0+\eta_1h)J_{-1}(h)+\sqrt{1 + h}(-h)^{\frac{4-n}{2}}P_{\frac{n-4}{2}}(h)
+\sqrt{1-\sqrt{-h}}(-h)^{\frac{5-n}{2}}P_{n-3}(\sqrt{-h}),
\end{split}
\end{equation*}
for $n\geq7$ odd,
\begin{equation*}
\begin{split}
M(h)
=&((-h)^{-\frac{1}{2}}-1)P_2(\sqrt{-h})
+(-h)^{\frac{4-n}{2}}(1+h)P_{\frac{n-5}{2}}(h)+(-h)^{\frac{5-n}{2}}\left(P_{\frac{n-5}{2}}(h)\bar I_2+\bar P_{\frac{n-5}{2}}(h)\bar I_0\right)\\
&+(\eta_0+\eta_1h)J_{-1}(h)+\sqrt{1+h}(-h)^{\frac{5-n}{2}}P_{\frac{n-5}{2}}(h)
+\sqrt{1-\sqrt{-h}}(-h)^{\frac{4-n}{2}}P_{n-2}(\sqrt{-h}),
\end{split}
\end{equation*}
where $J_{-1}(h)$ is given in \eqref{J4}, and $P_k$ and $\bar P_k$ represent polynomials of degree $k$.
\end{proposition}
\begin{proof}
The computations of $M(h)$ for $n\leq5$ are straightforward and thus we omit here. We mainly concern on the case $n\geq6$ even, and the case
$n\geq7$ odd can be obtained in a similar way.

For $n\geq6$ even,
\begin{equation*}
\begin{split}
M(h)=&\sum_{k=0}^{2n}\overline{m}_k(h)I_{k-6}(h)+
\sum_{k=0}^{2n+1}\widetilde{m}_k(h)J_{k-6}(h)\\
=&\sum_{i=0}^{n-1}\overline{m}_{2i+1}(h)I_{2i-5}(h)+\sum_{i=0}^{n}\overline{m}_{2i}(h)I_{2i-6}(h)+
\sum_{i=0}^{n}\widetilde{m}_{2i+1}(h)J_{2i-5}(h)+\sum_{i=0}^{n}\widetilde{m}_{2i}(h)J_{2i-6}(h).
\end{split}
\end{equation*}
Note that for $i\geq3$,
\begin{equation*}
\begin{split}
\overline{m}_{2i+1}(h)I_{2i-5}(h)=\left\{\begin{array}{ll}
(-h)^{\frac{3}{2}-i}(1+h)P_{i-2}(h) & \mbox{for $i\leq \frac{n-2}{2}$,} \\
(-h)^{\frac{5-n}{2}}(1+h)P_{\frac{n-6}{2}}(h) & \mbox{for $i\geq  \frac{n}{2}$,}
\end{array} \right.\\
\end{split}
\end{equation*}
\begin{equation*}
\begin{split}
\widetilde{m}_{2i+1}(h)J_{2i-5}(h)=\left\{\begin{array}{ll}
\sqrt{1+h}(-h)^{2-i}P_{i-2}(h) & \mbox{for $i\leq \frac{n}{2}$,} \\
\sqrt{1+h}(-h)^{\frac{4-n}{2}}P_{\frac{n-4}{2}}(h) & \mbox{for $i\geq  \frac{n+2}{2}$,}
\end{array} \right.\\
\end{split}
\end{equation*}
and for $i\geq4$,
\begin{equation*}
\begin{split}
\overline{m}_{2i}(h)I_{2i-6}(h)=\left\{\begin{array}{ll}
(-h)^{2-i}\left(P_{[\frac{i-4}{2}]+[\frac{i}{2}]}(h)\bar I_2+P_{i-2}(h)\bar I_0\right) & \mbox{for $i\leq \frac{n}{2}$,} \\
(-h)^{\frac{4-n}{2}}\left(P_{[\frac{i-4}{2}]+\frac{n}{2}-[\frac{i+1}{2}]}(h)\bar I_2+P_{\frac{n-4}{2}}(h)\bar I_0\right) & \mbox{for $i\geq  \frac{n+2}{2}$,}
\end{array} \right.\\
\end{split}
\end{equation*}
\begin{equation*}
\begin{split}
\widetilde{m}_{2i}(h)J_{2i-6}(h)=\left\{\begin{array}{ll}
\sqrt{(1-\sqrt{-h})}(-h)^{\frac{5}{2}-i}
P_{2[\frac{i}{2}]+i-3}(\sqrt{-h}) & \mbox{for $i\leq \frac{n}{2}$,} \\
\sqrt{(1-\sqrt{-h})}(-h)^{\frac{5-n}{2}}P_{n-3+2[\frac{i}{2}]-i}(\sqrt{-h}) & \mbox{for $i\geq  \frac{n+2}{2}$.}
\end{array} \right.\\
\end{split}
\end{equation*}
Therefore, for $n\geq6$ even,
\begin{equation*}
\begin{split}
M(h)=&\sum_{i=0}^{(n-2)/2}\overline{m}_{2i+1}(h)I_{2i-5}(h)
+\sum_{i=n/2}^{n-1}\overline{m}_{2i+1}(h)I_{2i-5}(h)\\
&+\sum_{i=0}^{n/2}\overline{m}_{2i}(h)I_{2i-6}(h)
+\sum_{i=(n+2)/2}^{n}\overline{m}_{2i}(h)I_{2i-6}(h)\\
&+\sum_{i=0}^{n/2}\widetilde{m}_{2i+1}(h)J_{2i-5}(h)
+\sum_{i=(n+2)/2}^{n}\widetilde{m}_{2i+1}(h)J_{2i-5}(h)\\
&+\sum_{i=0}^{n/2}\widetilde{m}_{2i}(h)J_{2i-6}(h)
+\sum_{i=(n+2)/2}^{n}\widetilde{m}_{2i}(h)J_{2i-6}(h)\\
=&(1-\sqrt{-h})(\delta_{-1}(-h)^{-\frac{1}{2}}+\delta_0+\delta_1\sqrt{-h})
+(-h)^{\frac{5-n}{2}}(1+h)P_{\frac{n-6}{2}}(h)\\[1ex]
&+h^{-1}(P_1(h)\bar I_0+\bar P_1(h)\bar I_2)
+(-h)^{\frac{4-n}{2}}\left(P_{\frac{n-4}{2}}(h)\bar I_2+\bar P_{\frac{n-4}{2}}(h)\bar I_0\right)\\[1ex]
&+\gamma_0\sqrt{1 + h}+(\eta_0+\eta_1h)J_{-1}(h)+\sqrt{1+h}(-h)^{\frac{4-n}{2}}P_{\frac{n-4}{2}}(h)\\[1ex]
&+\sqrt{1-\sqrt{-h}}(-h)^{-\frac{1}{2}}P_3(\sqrt{-h})
+\sqrt{1-\sqrt{-h}}(-h)^{\frac{5-n}{2}}
P_{n-3}(\sqrt{-h})\\[1ex]
=&((-h)^{-\frac{1}{2}}-1)P_2(\sqrt{-h})
+(-h)^{\frac{5-n}{2}}(1+h)P_{\frac{n-6}{2}}(h)+(-h)^{\frac{4-n}{2}}\left(P_{\frac{n-4}{2}}(h)\bar I_2+\bar P_{\frac{n-4}{2}}(h)\bar I_0\right)\\[1ex]
&+(\eta_0+\eta_1h)J_{-1}(h)+\sqrt{1 + h}(-h)^{\frac{4-n}{2}}P_{\frac{n-4}{2}}(h)
+\sqrt{1-\sqrt{-h}}(-h)^{\frac{5-n}{2}}P_{n-3}(\sqrt{-h}).
\end{split}
\end{equation*}
\end{proof}
Now, we begin to estimate the number of zeros of $M(h)$ obtained above.
By the definition of $\bar I_0(h)$ (see \eqref{barIk}), it is easy to know that
\[
\bar I_0(h)=-\frac{1}{2}\oint_{u^4-2u^2+y^2=h}y\mathrm{d}u=\frac{1}{2}\iint_{u^4-2u^2+y^2\leq h}\mathrm{d}\sigma>0
\]
in $h\in(-1,0)$. Let
\begin{equation}
v(h)=\frac{\bar I_2(h)}{\bar I_0(h)},
\end{equation}
then by \eqref{barIk} and Picard-Fuchs equation \eqref{I02}, we have
\begin{lemma}\label{le:v}
$v=v(h)$ is the solution of the differential system
\begin{equation}\label{vh}
\dot h=4h(1+h),\quad \dot v=-h+2(-2+h)v+5v^2,
\end{equation}
satisfying $v'(h)<0$ for $h\in(-1,0)$, $\lim_{h\rightarrow-1^+}v(h)=1$ and $\lim_{h\rightarrow0^-}v(h)=4/5$.
\end{lemma}
The proof of Lemma \ref{le:v} is given in Appendix A.3.
\begin{proposition}
$H_4(0)=1$ and $H_4(1)=4$.
\end{proposition}
\begin{proof}
For $n=0$, let $f_1=\sqrt{1+h},\ f_2=\bar I_2$.
Using
\begin{equation*}\begin{split}
W_1=&\sqrt{1+h}>0,\\
W_2=&\frac{
\bar I_0}{4\sqrt{1+h}}(3v(h)-1)>0,\\
\end{split}\end{equation*}
in $h\in(0,1)$, we know that $M(h)$ has at most one zero in $(-1,0)$ for $n=0$ by \emph{Chebyshev criterion}.

For $n=1$, let
\[
f_1=\sqrt{1+h}, \quad f_2=(2+\sqrt{-h})\sqrt{1-\sqrt{-h}}, \quad
f_3=h\sqrt{1-\sqrt{-h}},\quad
f_4=1+h, \quad f_5=\bar I_2.
\]
We can get
\begin{equation*}\begin{split}
W_1=&\sqrt{1+h}>0,\\
W_2=&-\frac{
(1-\sqrt{-h})^{3/2}}{4\sqrt{1+h}}<0,\\
W_3=&\dfrac{3(1-\sqrt{-h})^{2}}{32(1+\sqrt{-h})\sqrt{-h(1+h)}}>0,\\
W_4=&\dfrac{3(-7+\sqrt{-h})}{1024h(1+\sqrt{-h})^2\sqrt{-h(1+h)}}
>0,\\
W_5=&\dfrac{9(1-\sqrt{-h})(-5+2h-h\sqrt{-h})\bar I_0}{131072h^{4}(1+h)^{11/2}}
\left(\dfrac{40+7\sqrt{-h}-27h+17h\sqrt{-h}+3h^2}{-5+2h-h\sqrt{-h}}+10v(h)\right)>0,
\end{split}\end{equation*}
thus by \emph{Chebyshev criterion}, $M(h)$ has at most four zeros in $(-1,0)$ for $n=1$.

In fact, denote the function in the parenthesis of $W_5$ by $Z_5$. It is easy to verify that $Z_5$ is not a monotonous function
in $(-1,0)$. From \eqref{vh},
\[
Z'_5|_{Z_5=0}=\dfrac{-140+251\sqrt{-h}+642h+332(-h)^{3/2}+70h^2-15(-h)^{5/2}}
{8\sqrt{-h}(-5+2h-h\sqrt{-h})^2}
<0,\]
moreover, we have the asymptotic expansions of $Z_5$ near $h=-1$ and $h=0$:
\begin{equation*}\begin{split}
Z_5&=-\dfrac{1}{4}(h+1)-\dfrac{41}{384}(h+1)^2+o\left((h+1)^2\right),\quad\quad h\rightarrow-1^+,\\
Z_5&=-\dfrac{7}{5}\sqrt{-h}-\dfrac{1}{10}h(-82+90\log2-15\log(-h))+o\left(-h\right),\quad\quad h\rightarrow0^-.
\end{split}\end{equation*}
Since $\lim_{h\rightarrow-1^+}Z_5=\lim_{h\rightarrow0^-}Z_5=0$ and $Z_5$ is always negative near $h\rightarrow-1^+$ and $h\rightarrow0^-$, if $Z_5$ has zeros in $(-1,0)$, then it has at least two (taking into account their multiplicity), and one of them satisfies $Z'_5\geq0$, which contradicts with $Z'_5|_{Z_5=0}<0$. Thus, $Z_5$ has none zeros in $(-1,0)$, which implies that $W_5$ is positive in $(-1,0)$.
\end{proof}
\begin{proposition}
$H_4(2)=7$.
\end{proposition}
\begin{proof}
For $n=2$, notice that by the change
\[
\cos\theta=\frac{1-u^2}{\sqrt{1+h}}, \quad \theta\in[0,\pi],\]
\begin{equation}\label{barI01}
\begin{split}
&\bar I'_0=\int_{u_1}^{u_2}\frac{1}{2\sqrt{h+2u^2-u^4}}\mathrm{d}u
=\frac{1}{4}\displaystyle\int_0^{\pi}\frac{1}{\sqrt{1-\sqrt{1+h}\cos\theta}}\mathrm{d}\theta
=\frac{1}{4}J(\sqrt{1+h}),\\
&\bar I'_2=\int_{u_1}^{u_2}\frac{u^2}{2\sqrt{h+2u^2-u^4}}\mathrm{d}u
=\frac{1}{4}\displaystyle\int_0^{\pi}\sqrt{1-\sqrt{1+h}\cos\theta}\mathrm{d}\theta
=\frac{1}{4}I(\sqrt{1+h}),
\end{split}
\end{equation}
where
\begin{equation}\begin{split}\label{IJ}
I(r)&=\displaystyle\int_0^{\pi}\sqrt{1-r\cos\theta}\mathrm{d}\theta
=2\sqrt{1+r}E\left(\frac{2r}{1+r}\right),\\
J(r)&=\displaystyle\int_0^{\pi}\frac{1}{\sqrt{1-r\cos\theta}}\mathrm{d}\theta
=\frac{2}{\sqrt{1+r}}K\left(\frac{2r}{1+r}\right)
\end{split}\end{equation}
are the functions defined in \cite{CLZ}. Then use \eqref{I02}, and let
$r=\sqrt{1+h}$,
\begin{equation}\begin{split}
F(r)=M(h)|_{h=r^2-1}=rf(r)+k_0f_0,
\end{split}\end{equation}
where $f_0=r$ and $f(r)$ is the averaged function obtained in \cite{CLZ}, with
\begin{equation}\begin{split}
&k_0=\frac{1}{2}(3\sqrt{2}\beta_0-\sqrt{2}\beta_1+2\gamma_0-2\eta_0)=a^+_{00}-a^-_{00},\\
&k_1=\delta_0,\quad k_2=
-\frac{\beta_0-\beta_1}{\sqrt{2}},\quad k_3=\frac{\beta_1}{\sqrt{2}},\quad k_4=\frac{\eta_0}{2}, \quad
k_5=\frac{\alpha_0}{5},\quad k_6=\frac{\alpha_0+5\xi_0}{15}.
\end{split}\end{equation}
We exploit the same approach than in \cite{CLZ}, and eliminate the logarithm function first. Since
\begin{equation}
\left(\dfrac{F(r)}{1-r^2}\right)'=\dfrac{G(r)}{(1-r^2)^2},
\end{equation}
and $G(r)$ has as many zeros as $F(r)$ in $r\in(0,1)$, we consider $G(r)$ in the following, where
\begin{equation}\label{G}
G(r)=m_1 g_1+m_2 g_2+m_3 g_3+m_4 g_4+m_5 g_5+m_6 g_6+m_7g_7,
\end{equation}
with
\begin{equation}\begin{split}\label{g}
g_1&=r,\\
g_2&=r^2,\\
g_3&=6+4r(\sqrt{1-r}-\sqrt{1+r})-(3+r^2)(\sqrt{1-r}+\sqrt{1+r}),\\
g_4&=-4+2(1+r^2)(\sqrt{1-r}+\sqrt{1+r})+r(-5+r^2)(\sqrt{1-r}-\sqrt{1+r}),\\
g_5&=r\left((-5+r^2)I(r)+(1-r^2)J(r)\right),\\
g_6&=r\left(4I(r)-(1-r^2)J(r)\right),\\
g_7&=1,
\end{split}\end{equation}
and
\begin{equation*}\begin{split}
m_1=2k_1,\,\,\, m_2=3k_2-2k_3+4k_4+k_0,\,\,\, m_3=\dfrac{k_2}{2},\,\,\,m_4=\dfrac{k_3}{2},\,\,\, m_5=-\dfrac{k_5}{2},\,\,\, m_6=\dfrac{k_6}{2},\,\, m_7=k_0.
\end{split}\end{equation*}
By the results in \cite{CLZ}, we know that the Wronskian determinants $W_1, W_2, \cdots, W_6$ on the ordered set $(g_1, g_2,\cdots, g_7)$ do not vanish in $r\in(0,1)$. Thus it suffices to consider the Wronskian determinant
\[
W_7=-\frac{6075 (1-s)J^2(\sqrt{1-s^2})}{8192 s^{18} \left(1-s^2\right)^3}\left(Y_{70}+Y_{71}w(s)+Y_{72}w^2(s)\right),
\]
where $s=\sqrt{1-r^2}$,
\begin{equation}\begin{split}
Y_{70}=&s^2 \left(-960-960 s+305136 s^2+314496 s^3-291576 s^4-288351 s^5+28820 s^6\right.\\
&\left.-1205 s^7+4170 s^8+46375 s^9+14000 s^{10}+525 s^{11}+1050 s^{12}\right),\\
Y_{71}=&2 \left(1920+1920
   s-17184 s^2-35424 s^3-464928 s^4-469008 s^5+364162 s^6\right.\\
   &\left.+396377 s^7-37560 s^8-110965 s^9-44870 s^{10}+14175 s^{11}+6300 s^{12}+ 525 s^{13}\right),\\
Y_{72}=&-18624-22464 s+101040 s^2+125280 s^3+524168 s^4+592453 s^5-259404 s^6\\
&-365129 s^7-104350 s^8-15645 s^9+84000 s^{10}+19425 s^{11}-3150 s^{12},
\end{split}\end{equation}
and $w(s)=I(r)/J(r)|_{r=\sqrt{1-s^2}}$ is the solution of the following differential system
\begin{equation}\label{ws}
\dot{s}=2s(1-s^2),\quad \dot{w}=s^2-2s^2w+w^2,
\end{equation}
satisfying $w'(s)>0$, and
\[
\lim_{s\rightarrow 0^+} w(s)=0, \quad \quad \lim_{s\rightarrow1^-}w(s)=1.
\]
Define
\[
\Psi(s,w)=Y_{70}+Y_{71}w+Y_{72}w^2.
\]
Then the number of zeros of $W_7$ in $(0,1)$ equals the number of intersection points of the
curve $C=\{(s,w)|\Psi(s,w)=0,s\in(0,1)\}$ and the curve $\Gamma=\{w=w(s),s\in(0,1)\}$
in the $(s,w)$-plane.

First, we show that the curve $C$ and $\Gamma$ can intersect at least one point.
It is easy to know that $Y_{72}$ has a unique zero $s_0$ in $(0,1)$ by Sturm Theorem, and $17/50<s_0<7/20$. If
$s=s_0$, then $\Psi=0$ implies that $w_0=-Y_{70}(s_0)/Y_{71}(s_0)$. In the following, we consider $\Psi$ in $s\in (0,s_0)\bigcup(s_0,1)$.
Let $C_{+}$ and $C_{-}$ be two branches of the curve $C$, denoted by
\begin{equation}
C_{+}=\{w_{+}(s)=\dfrac{-Y_{71}+\sqrt{\Delta(s)}}{2Y_{72}}\}, \quad
C_{-}=\{w_{-}(s)=\dfrac{-Y_{71}-\sqrt{\Delta(s)}}{2Y_{72}}\},
\end{equation}
where $\Delta(s)=Y_{71}^2-4Y_{70}Y_{72}>0$ in $s\in(0,1)$ by Sturm Theorem, with
\begin{equation*}\begin{split}
\Delta(s)=&900 (1 + s)^4 (16384 - 32768 s - 323584 s^2 + 352256 s^3 +
   19012608 s^4 - 29902848 s^5\\
   & - 47736576 s^6 + 165632640 s^7 +
   3358208 s^8 - 270031520 s^9 + 144902688 s^{10}\\
   & + 18987696 s^{11} +
   278405247 s^{12} - 359686304 s^{13} + 66366873 s^{14} + 15590624 s^{15}\\
   & +11191110 s^{16} + 12549152 s^{17} - 8746430 s^{18} - 960400 s^{19} +
   376075 s^{20}\\
   & - 117600 s^{21} + 15925 s^{22}).
\end{split}\end{equation*}
A direct computation shows that
\begin{equation*}\begin{split}
&\lim_{s\rightarrow 0^+} w_{+}(s)=0, \quad \lim_{s\rightarrow s_0^-} w_{+}(s)=-\infty, \quad \lim_{s\rightarrow s_0^+}w_{+}(s)=+\infty,
\quad \lim_{s\rightarrow1^-}w_{+}(s)=1,\\
& \lim_{s\rightarrow 0^+}w_{-}(s)=\frac{20}{97}, \quad \lim_{s\rightarrow s_0^-} w_{-}(s)=\lim_{s\rightarrow s_0^+} w_{-}(s)=w_0,  \quad \lim_{s\rightarrow1^-}w_{-}(s)=\frac{1}{5},
\end{split}\end{equation*}
which implies that $w_{+}(s)$ is not continuous in $s_0$, while $w_{-}(s)$ is continuous in $s_0$, and thus is  continuous in $(0,1)$.
Further, $w_{+}(s)$ and $w(s)$ have the asymptotic expansions when $s\rightarrow0^+$,
\begin{equation*}\begin{split}
w_{+}(s)=&\frac{1}{4}s^2 - \frac{4923}{64}s^4+o(s^4),\\
w(s)=&\frac{4}{6\log2-2\log s}+o\left(\frac{1}{6\log2-2\log s}\right),
\end{split}\end{equation*}
and when $s\rightarrow1^-$,
\begin{equation*}\begin{split}
w_{+}(s)=&1-\frac{1}{2}(1-s)+\frac{7}{16}(1-s)^2+\frac{22407}{8768}(1-s)^3+o((1-s)^3),\\
w(s)=&1-\frac{1}{2}(1-s)-\frac{1}{32}(1-s)^2-\frac{3}{128}(1-s)^3+o((1-s)^3).
\end{split}\end{equation*}
Comparing these results, we have
\begin{equation*}\begin{split}\label{0}
&w_{-}(s)>w(s)>w_{+}(s),\quad \mbox{$s\rightarrow0^+$},\\
&w_{+}(s)>w(s)>w_{-}(s), \quad \mbox{$s\rightarrow1^-$},
\end{split}\end{equation*}
and hence the curve $\Gamma$ intersects $C_{-}$ at least one point $(s^*,w^*)$.

Next, we show that there exist exactly two points of $C$ at which the vector field \eqref{ws} is tangent on $C$. We call them contact points.
A direct computation shows that
\[
\Phi(s,w)=\left(\frac{\partial \Psi(s,w)}{\partial s},\frac{\partial \Psi(s,w)}{\partial w}\right)\cdot(\dot{s},\dot{w})=-2\sum_{k=0}^{k=3}\phi_k(s)w^k,
\]
where
\begin{equation*}\begin{split}
\phi_0(s)=& s^3 (960 - 1205280 s - 1539936 s^2 + 3434928 s^3 +
     4059945 s^4 - 2344178 s^5\\
     & - 2403989 s^6 + 226420 s^7 -
     410005 s^8 - 81430 s^9 + 489125 s^{10}\\
     & + 147000 s^{11} + 6300 s^{12} +
     14700 s^{13}),\\
\phi_1(s)=& s(-3840 + 91200 s + 242688 s^2 + 3515280 s^3 + 4281408 s^4 -
     9543392 s^5\\& - 11769827 s^6 + 5958632 s^7 + 8704531 s^8 +
     325670 s^9 - 2515505 s^{10}\\& - 1222340 s^{11} + 307125 s^{12} +
     166950 s^{13} + 14700 s^{14}),\\
\phi_2(s)=& -1920 + 20544 s - 222144 s^2 - 407808 s^3 - 1227584 s^4 -
 1866857 s^5 \\&+ 4337270 s^6 + 6306697 s^7 - 1202872 s^8 -
 3034391 s^9 - 1838630 s^{10}\\& - 399945 s^{11} + 1039500 s^{12} +
 252000 s^{13} - 44100 s^{14},\\
\phi_3(s)=&-Y_{72}.
\end{split}\end{equation*}
Obviously, $\phi_3(s)$ does not vanish in $(0,s_0)\bigcup(s_0,1)$. Thus, the resultant $R$
of $\Psi(s,w)$ and $\Phi(s,w)$ with respect to $w$, has the form $R=810000(1-s)^2s^4(1+s)^{10}Y_{72}R_1(s)R_2(s)$,
where
\begin{equation*}\begin{split}
R_1(s)=&-15360 - 30720 s + 690176 s^2 +
   257920 s^3 + 800384 s^4 + 1928800 s^5- 1741120 s^6\\&  -
   1638704 s^7 + 1429155 s^8 + 79254 s^9 - 266350 s^{10} - 42140 s^{11} +
   42875 s^{12} + 7350 s^{13},\\
R_2(s)=&10764288 - 43057152 s + 1390657536 s^2 -
   5132058624 s^3 + 1851543552 s^4 \\& + 16976596992 s^5-
   22820136960 s^6 - 11403103872 s^7 + 17152685568 s^8 \\& +
   58829080800 s^9 - 117428834304 s^{10} + 82319936688 s^{11} -
   32444436073 s^{12}\\& - 8041553870 s^{13} + 49398619998 s^{14} -
   23623365500 s^{15} - 20007160159 s^{16}\\& + 13090741902 s^{17} +
   90897156 s^{18} - 1222410840 s^{19} + 772710057 s^{20} +
   175441070 s^{21}\\& - 135536450 s^{22} + 7923300 s^{23} + 3301375 s^{24} +
   120050 s^{25}.
\end{split}\end{equation*}
By Sturm Theorem, $R_1(s)$ has a unique zero $s_1$ in $(4/25,17/100)$, and
$R_2(s)$ has a unique zero $s_2$ in $(12/25,49/100)$, which means that there
exist $w_1$ and $w_2$, such that
\[
\Psi(s_1,w_1)=\Phi(s_1,w_1)=\Psi(s_2,w_2)=\Phi(s_2,w_2)=0.
\]
Besides,
\[
\lim_{s\rightarrow s_0^+}\left(w'_{-}(s)-\frac{\mathrm{d}w}{\mathrm{d}s}\Big|_{w=w_{-}(s)}\right)=
\lim_{s\rightarrow s_0^-}\left(w'_{-}(s)-\frac{\mathrm{d}w}{\mathrm{d}s}\Big|_{w=w_{-}(s)}\right)\approx0.34\neq0.
\]
This confirms that there are exactly two points of the curve $C$ at which the vector field \eqref{ws} is tangent to $C$.

To prove that $\Psi(s,w(s))$ has a unique zero in $s\in(0,1)$, we introduce an
auxiliary straight line $w=2/5$. By the Sturm Theorem,
\begin{equation}\begin{split}
\Psi(s,2/5)=&\frac{1}{25}(-36096 - 51456 s + 36480 s^2 - 231360 s^3 + 426512 s^4 +
   852052 s^5\\& - 1043776 s^6 - 741751 s^7 - 448100 s^8 - 2312005 s^9 -
   457150 s^{10}\\& + 1520575 s^{11}
    + 463400 s^{12} + 23625 s^{13} + 26250 s^{14})<0,
\end{split}\end{equation}
which demonstrates that the straight line $w=2/5$ is above the curve $C_{-}$ in $(0,1)$. It follows from the values of the endpoints of
$w(s)$ and the monotonicity of $w(s)$, that the straight line $w=2/5$ and the curve $\Gamma$ will intersect at a unique point $(s_*,2/5)$
 with $1/25<s_*<1/10$. Thus $w_{-}(s_*)<2/5=w(s_*)$, and $w(s)>2/5>w_{-}(s)$ holds in $s\in(s_*,1)$, which shows that
the curves $\Gamma$ and $C_{-}$ intersect at least one point $(s^*,w^*)$ when $s\in(0,s_*)$ and can not intersect when $s\in(s_*,1)$. Since
\begin{equation*}\begin{split}
&\lim_{s\rightarrow0^+}\left(w'_{-}(s)-\frac{\mathrm{d}w}{\mathrm{d}s}\Big|_{w=w_{-}(s)}\right)=-\infty,\\
&\lim_{s\rightarrow\frac{3}{10}}\left(w'_{-}(s)-\frac{\mathrm{d}w}{\mathrm{d}s}\Big|_{w=w_{-}(s)}\right)\approx0.39,\\
&\lim_{s\rightarrow1^-}\left(w'_{-}(s)-\frac{\mathrm{d}w}{\mathrm{d}s}\Big|_{w=w_{-}(s)}\right)=-\infty,
\end{split}\end{equation*}
there exist two points on $C_{-}$ at which the vector field \eqref{ws} is tangent to the curve $C_{-}$.
By the result above, the curve $\Gamma$ can not intersect $C_{-}$ and $C_{+}$ in other point, otherwise,
extra contact point will emerge, which results in a contradiction.
Hence, $\Psi(s,w(s))$, as well as $W_7$, has a unique zero in $s\in(0,1)$.

Finally, note that $s^*<s_*<s_1<s_2$, thus the unique zero of $W_7$ is simple. Combined with
$W_i\not\equiv0, i=1,2,\cdots,6$, it follows from Lemma \ref{le:NT}
that there exists a linear combination of $G(r)$ such that $G(r)$ has exactly $7$ zeros.
Thus, $F(r)$ can have at most $7$ zeros in $r\in(0,1)$, which is equivalent to $M(h)$ having at most $7$ zeros in $h\in(-1,0)$ for $n=2$.
\end{proof}

\begin{figure}[h]
\centering
\includegraphics[width=.45\textwidth]{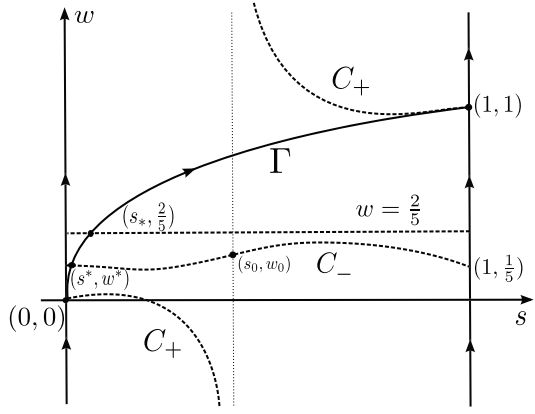}
\caption{\small{The curve $\Gamma$ has a unique common point with $C_{-}$.}}
\label{fig}
\end{figure}

For $n\geq3$, we eliminate the different kinds of functions by taking derivatives. First, we get rid of the logarithm
function by taking a second order derivative and then classify the derived function $M''(h)$ into four kinds of functions. Next, we eliminate two kinds of functions which include polynomials of $h$ as numerators by multiplying nonzero factors and taking derivatives.  Finally,  it suffices to consider the derived function, which is
described in more detail in the following. Take the case $n\geq6$ even for example.

\noindent$(i)$ Eliminate the logarithm function
\begin{equation}\begin{split}\label{d1}
M''
=&\dfrac{P_{\frac{n-4}{2}}(h)}{(-h)^{\frac{n-1}{2}}}
+\dfrac{P_{\frac{n-2}{2}}(h)\bar I_2+\bar P_{\frac{n-2}{2}}(h)\bar I_0}{(-h)^{\frac{n}{2}}(1+h)}+\dfrac{P_{\frac{n}{2}}(h)}{(1+h)^{\frac{3}{2}}(-h)^{\frac{n}{2}}}
+\dfrac{P_{n-1}(\sqrt{-h})}{(1-\sqrt{-h})^{\frac{3}{2}}(-h)^{\frac{n-1}{2}}},
\end{split}\end{equation}
$(ii)$ eliminate the first part of $M''$ by induction
\begin{equation}\begin{split}\label{d2}
F=&\left((-h)^{\frac{n-1}{2}}M''\right)^{(\frac{n-2}{2})}\\
=&
\dfrac{P_{n-2}(h)\bar I_2+\bar P_{n-2}(h)\bar I_0}{(-h)^{\frac{n-1}{2}}(1+h)^{\frac{n}{2}}}
+\dfrac{P_{\frac{n}{2}}(h)}{(1+h)^{\frac{n+1}{2}}(-h)^{\frac{n-1}{2}}}
+\dfrac{P_{\frac{3n}{2}-3}(\sqrt{-h})}{(1-\sqrt{-h})^{\frac{n+1}{2}}(-h)^{\frac{n-3}{2}}},
\end{split}\end{equation}
$(iii)$ eliminate the second part of $F$ by induction
\begin{equation}\begin{split}\label{d3}
G=&\left((1+h)^{\frac{n+1}{2}}(-h)^{\frac{n-1}{2}}F\right)^{(\frac{n+2}{2})}\\
=&\left(\sqrt{1+h}(P_{n-2}(h)\bar I_2+\bar P_{n-2}(h)\bar I_0)
+(-h)(1+\sqrt{-h})^{\frac{n+1}{2}}P_{\frac{3n}{2}-3}(\sqrt{-h})\right)^{(\frac{n+2}{2})}\\
=&
\dfrac{P_{\frac{3n}{2}-1}(h)\bar I_2+\bar P_{\frac{3n}{2}-1}(h)\bar I_0}{(-h)^{\frac{n+2}{2}}(1+h)^{\frac{n+1}{2}}}
+\dfrac{P_{2n-3}(\sqrt{-h})}{(1+\sqrt{-h})^{\frac{1}{2}}(-h)^{\frac{n-1}{2}}}.
\end{split}\end{equation}
Notice that $M(-1)=0$, thus,
\begin{equation}\begin{split}\label{N1}
H_4(n)&\leq \#\{-1<h<0|G(h)=0\}+\dfrac{n+2}{2}+\dfrac{n-2}{2}+2-1\\
&\leq \#\{-1<h<0|G(h)=0\}+n+1,
\end{split}\end{equation}
where $\#$ denotes the number of elements of a finite set.

We need consider the zeros of $G$ in $(-1,0)$ and we will give a rough estimate of zeros of $G$ using the method in \cite{LLLZ}. Let
\begin{equation}
\begin{split}
G_1&=\dfrac{P_{2n-3}(\sqrt{-h})}{(1+\sqrt{-h})^{\frac{1}{2}}(-h)^{\frac{n-1}{2}}},\\
G_2&=\dfrac{P_{\frac{3n}{2}-1}(h)}{(-h)^{\frac{n+2}{2}}(1+h)^{\frac{n+1}{2}}},\\
G_0&=\dfrac{\bar P_{\frac{3n}{2}-1}(h)}{(-h)^{\frac{n+2}{2}}(1+h)^{\frac{n+1}{2}}}.
\end{split}
\end{equation}
Obviously, $G_1$ has at most $2n-3$ zeros in $(-1,0)$.
\begin{equation*}
\begin{split}
\left(\frac{G}{G_1}\right)'=U_2\bar I_2+U_0\bar I_0,
\end{split}
\end{equation*}
where
\begin{equation*}
\begin{split}
U_2&=\dfrac{1}{4G_1^2h(h+1)}(4 h (h+1)(G_1G_2'-G_2G_1')+5hG_1G_2-5G_0G_1)\\
&=\dfrac{(-h)^{-1/2-n}(1+\sqrt{-h})^{-1/2}(1+h)^{-(n+1)/2} P_{5n-4}(\sqrt{-h})}{4G_1^2h(h+1)},\\
U_0&=\dfrac{1}{4G_1^2h(h+1)}(4 h (h+1)(G_1G_0'-G_0G_1')+(3h+4)G_0G_1-hG_1G_2)\\
&=\dfrac{(-h)^{-1/2-n}(1+\sqrt{-h})^{-1/2}(1+h)^{-(n+1)/2}\bar P_{5n-3}(\sqrt{-h})}{4G_1^2h(h+1)}.\\
\end{split}
\end{equation*}
Let $g=P_{5n-4}(\sqrt{-h})\bar I_2+\bar P_{5n-3}(\sqrt{-h})\bar I_0$.
Then
\begin{equation}\label{N1G}
\#\{\left(\frac{G}{G_1}\right)'=0\}\leq\#\{g=0\}.
\end{equation}
Using the method in \cite{LLLZ}, consider the function
\[
U=\frac{\bar I_2}{\bar I_0}+\frac{\bar P_{5n-3}(\sqrt{-h})}{P_{5n-4}(\sqrt{-h})},
\]
and compute the number of zeros of $U$ by a Ricatti equation. Then we have
\[
\#\{g=0\}\leq 15n-8.
\]
With \eqref{N1} and \eqref{N1G},
\[
H_4(n)\leq 15n-8+1+2n-3+2n-3+n+2-1=20n-12.
\]
Similarly, for $n\geq7$ odd,
\[
H_4(n)\leq 15n-8+1+2n-2+2n-2+n+2-1=20n-10,
\]
and for $n=3,4,5$, $H_4(n)\leq 12n+4.
$
\section{Zeros of $M(h)$ for system $S_4$ in the smooth case}\label{sec:S4s}

This section is devoted to giving an improved result on the number of zeros of the first Melnikov function $M(h)$ for quadratic isochronous center $S_4$ in \cite{LLLZ}.

When the perturbation polynomials are smooth, by the result in Section \ref{sec:S4}, we have
\begin{equation*}
\begin{split}
M(h)=((-h)^{-\frac{1}{2}}-1)P_2(\sqrt{-h})
+(-h)^{\frac{5-n}{2}}(1+h)P_{\frac{n-6}{2}}(h)+(-h)^{\frac{4-n}{2}}\left(P_{\frac{n-4}{2}}\bar I_2+\bar P_{\frac{n-4}{2}}\bar I_0\right),
\end{split}
\end{equation*}
for $n\geq6$ even, and $n\geq7$ odd,
\begin{equation*}
\begin{split}
M(h)=&((-h)^{-\frac{1}{2}}-1)P_2(\sqrt{-h})
+(-h)^{\frac{4-n}{2}}(1+h)P_{\frac{n-5}{2}}(h)+(-h)^{\frac{5-n}{2}}\left(P_{\frac{n-5}{2}}\bar I_2+\bar P_{\frac{n-5}{2}}\bar I_0\right).
\end{split}
\end{equation*}
We will estimate the number of zeros  of $M(h)$
for $n\geq6$ even in the following. The case of $n\geq7$ odd can be obtained in a similar way.

Using the result in \eqref{d1} and \eqref{d2},
\begin{equation*}
\left((-h)^{\frac{n-1}{2}}M''\right)^{(\frac{n-2}{2})}=
\dfrac{P_{n-2}(h)\bar I_2+\bar P_{n-2}(h)\bar I_0}{(-h)^{\frac{n-1}{2}}(1+h)^{\frac{n}{2}}}.
\end{equation*}
Let $g=P_{n-2}(h)\bar I_2+\bar P_{n-2}(h)\bar I_0$, and notice that $M(-1)=0$, then
\begin{equation}\begin{split}\label{N2}
\#\{-1<h<0|M(h)=0\}&\leq \#\{g=0\}+\dfrac{n-2}{2}+2-1=\#\{g=0\}+\dfrac{n}{2}.
\end{split}\end{equation}
Let $\mathbf{I}(h)=(\bar I_0,\bar I_2)^\top$, then by \eqref{I02}, $\mathbf{I}(h)$ satisfies
a two-dimensional first-order Fuchsian system
\begin{equation}\label{Ih}
\mathbf{I}(h)=\mathbf{A}(h)\mathbf{I}'(h),
\end{equation}
where
\begin{equation*}\label{A}
\mathbf{A}(h)=
\left(\begin{array}{cc}\frac{4h}{3} &\frac{4}{3}\\[1ex] \frac{4h}{15} & \frac{4(4+3h)}{15} \end{array}\right).
\end{equation*}
It is easy to verify that system \eqref{Ih} satisfies the assumptions $(\mathrm{H1})$-$(\mathrm{H3})$, see Appendix A.4 or \cite{GI} for details. That is,
\vspace{-10pt}
\begin{itemize}
\item[(1)]$\mathbf{A}'=\left(\begin{array}{cc}\frac{4}{3} &0\\[1ex] \frac{4}{15} & \frac{4}{5} \end{array}\right)$ is a constant matrix which has real distinct eigenvalues $\frac{4}{3}$ and $\frac{4}{5}$.
\item[(2)]$\det \mathbf{A}(h)=\frac{16}{15}h(1+h)$ has real distinct zeros $h_0=-1, h_1=0$, and the identity
trace $\mathbf{A}(h)\equiv(\det \mathbf{A}(h))'=\frac{16}{15}(1+2h)$.
\item[(3)]$\mathbf{I}(h)$ is analytic in a neighborhood of $-1$.
\end{itemize}
Thus $\lambda=3/4$, and $\lambda^*=3/4$. It follows from Theorem \ref{th:GI} and $\dim g=2n-2$ that an upper bound of the number of zeros of
$g$ is $(2n-3)+1=2n-2$. Moreover, $-1$ is a trivial zero  of $g$, hence
\[
\#\{g=0\}\leq 2n-3,
\]
and it follows from \eqref{N2} that
\[
\#\{-1<h<0|M(h)=0\}\leq \frac{5n-6}{2}.
\]
Similarly, for $n\geq7$ odd,
\[
\#\{-1<h<0|M(h)=0\}\leq 2n-3+\frac{n-1}{2}+2-1=\frac{5n-5}{2}.
\]
Thus, we have that the upper bound of the number of zeros of $M(h)$ is $[\frac{5n-5}{2}]$.
Specially, this upper bound is applicable to the cases of $n=2,3,4,5$, except that the maximum number of
zeros of $M(h)$ is $i$ for $n=i, i=0,1$.

\section{Acknowledgements}
We appreciate the helpful suggestions of Professor Changjian Liu. The first author is partially supported by NSF of China (Grant No. 11401111, No. 11571379, No. 11571195 and No. 11601257). The second author is partially supported by NSF of China (Grant No. 11571195). The third author is partially supported by NSF of China (Grant No. 11231001 and No. 11371213).

\section*{Appendix}
\noindent{\bf{A.1\ Proof of Lemma \ref{lem:P}.}} We only consider the case when the functions have the form $P_n(x)/(x^{p}(a+x)^{q})$, and
the other case can be obtained in a similar way.

Let $P_n(x)=\sum_{k=0}^nd_k x^k$. Then
\begin{equation}\begin{split}
\displaystyle\frac{P_n(x)}{x^{p}(a+x)^{q}}=\sum_{k=0}^{n}d_k x^{k-p}(a+x)^{-q}.
\end{split}\end{equation}
\begin{equation}\begin{split}\label{xj}
&\left(x^{k-p}(a+x)^{-q}\right)^{(j)}\\
=&\displaystyle\sum_{i=0}^{j}C_{j}^{i}\left(x^{k-p}\right)^{(i)}\left((a+x)^{-q}\right)^{(j-i)}\\
=&\displaystyle\sum_{i=0}^{j}C_{j}^{i}\left(\prod_{m=0}^{i-1}(k-p-m)x^{k-p-i}\right)
\left(\prod_{m=0}^{j-i-1}(-q-m)(a+x)^{-q-j+i}\right)\\
=&\frac{1}{x^{p+j}(a+x)^{q+j}}\displaystyle\sum_{i=0}^{j}C_{j}^{i}\prod_{m=0}^{i-1}(k-p-m)\prod_{m=0}^{j-i-1}(-q-m)x^{k+j-i}(a+x)^{i}\\
=&\frac{1}{x^{p+j}(a+x)^{q+j}}\displaystyle\sum_{i=0}^{j}C_{j}^{i}\prod_{m=0}^{i-1}(k-p-m)\prod_{m=0}^{j-i-1}(-q-m)x^{k+j-i}\sum_{s=0}^{i}C_{i}^{s}a^{s}x^{i-s}\\
=&\frac{1}{x^{p+j}(a+x)^{q+j}}\displaystyle\sum_{i=0}^{j}C_{j}^{i}\prod_{m=0}^{i-1}(k-p-m)\prod_{m=0}^{j-i-1}(-q-m)\sum_{s=0}^{i}C_{i}^{s}a^{s}x^{k+j-s}\\
=&\frac{1}{x^{p+j}(a+x)^{q+j}}\displaystyle\sum_{s=0}^{j}C_{k+j-s}x^{k+j-s}.
\end{split}\end{equation}
We claim that for fixed $0\leq s\leq j$, the coefficient of $x^{k+j-s}$, denoted by $C_{k+j-s}$, satisfies
\begin{equation}\label{ckjs}
C_{k+j-s}=a^sC_{j}^{s}\prod_{m=s}^{j-1}(k-p-q-m)\prod_{m=0}^{s-1}(k-p-m),
\end{equation}
where we set $$\prod_{m=i}^{i-1}(y-m)=1.$$
Obviously, by the last equality of \eqref{xj},
\begin{equation}\begin{split}
C_{k+j-s}=&a^s\displaystyle\sum_{i=s}^{j}C_{j}^{i}\prod_{m=0}^{i-1}(k-p-m)\prod_{m=0}^{j-i-1}(-q-m)C_{i}^{s}\\
=&a^s\displaystyle\sum_{i=s}^{j}C_{j}^{s}C_{j-s}^{i-s}\prod_{m=0}^{i-1}(k-p-m)\prod_{m=0}^{j-i-1}(-q-m)\\
=&a^sC_{j}^{s}\displaystyle\sum_{i=0}^{j-s}C_{j-s}^{i}\prod_{m=0}^{i+s-1}(k-p-m)\prod_{m=0}^{j-i-s-1}(-q-m)\\
=&a^sC_{j}^{s}\prod_{m=0}^{s-1}(k-p-m)\displaystyle\sum_{i=0}^{j-s}
C_{j-s}^{i}\prod_{m=s}^{i+s-1}(k-p-m)\prod_{m=0}^{j-i-s-1}(-q-m).
\end{split}\end{equation}
Thus, it suffices to prove that
\begin{equation}\label{kpq}
\prod_{m=s}^{j-1}(k-p-q-m)
=\displaystyle\sum_{i=0}^{j-s}C_{j-s}^{i}\prod_{m=s}^{i+s-1}(k-p-m)\prod_{m=0}^{j-i-s-1}(-q-m).
\end{equation}
We exploit the method of induction. It is easy to verify that \eqref{kpq} holds for $j=s$ and $j=s+1$. Suppose that
\eqref{kpq} holds for $j=l$, then when $j=l+1$,
\begin{equation*}\begin{split}
&\prod_{m=s}^{l}(k-p-q-m)\\
=&(k-p-q-l)\prod_{m=s}^{l-1}(k-p-q-m)\\
=&(k-p-q-l)\displaystyle\sum_{i=0}^{l-s}C_{l-s}^{i}\prod_{m=s}^{i+s-1}(k-p-m)\prod_{m=0}^{l-i-s-1}(-q-m)\\
=&\displaystyle\sum_{i=0}^{l-s}C_{l-s}^{i}(k-p-(i+s)-q-(l-i-s))\prod_{m=s}^{i+s-1}(k-p-m)\prod_{m=0}^{l-i-s-1}(-q-m)\\
=&\displaystyle\sum_{i=0}^{l-s}C_{l-s}^{i}(k-p-(i+s))\prod_{m=s}^{i+s-1}(k-p-m)\prod_{m=0}^{l-i-s-1}(-q-m)\\
&+\displaystyle\sum_{i=0}^{l-s}C_{l-s}^{i}(-q-(l-i-s))\prod_{m=s}^{i+s-1}(k-p-m)\prod_{m=0}^{l-i-s-1}(-q-m)\\
=&\displaystyle\sum_{i=0}^{l-s}C_{l-s}^{i}\prod_{m=s}^{i+s}(k-p-m)\prod_{m=0}^{l-i-s-1}(-q-m)
+\displaystyle\sum_{i=0}^{l-s}C_{l-s}^{i}\prod_{m=s}^{i+s-1}(k-p-m)\prod_{m=0}^{l-i-s}(-q-m)\\
=&\prod_{m=s}^{l}(k-p-m)
+\displaystyle\sum_{i=1}^{l-s}C_{l-s}^{i-1}\prod_{m=s}^{i+s-1}(k-p-m)\prod_{m=0}^{l-i-s}(-q-m)\\
&+\displaystyle\sum_{i=0}^{l-s}C_{l-s}^{i}\prod_{m=s}^{i+s-1}(k-p-m)\prod_{m=0}^{l-i-s}(-q-m)\\
=&\prod_{m=s}^{l}(k-p-m)
+\displaystyle\sum_{i=1}^{l-s}(C_{l-s}^{i-1}+C_{l-s}^{i})\prod_{m=s}^{i+s-1}(k-p-m)\prod_{m=0}^{l-i-s}(-q-m)+
\prod_{m=0}^{l-s}(-q-m)\\
=&\displaystyle\sum_{i=0}^{l+1-s}C_{l+1-s}^{i}\prod_{m=s}^{i+s-1}(k-p-m)\prod_{m=0}^{l-i-s}(-q-m),
\end{split}\end{equation*}
which means that \eqref{kpq} holds for $j=l+1$. It follows from the method of induction that \eqref{kpq} holds, and that the coefficient expression \eqref{ckjs} is true. Thus, for $j=n+1-(p+q)$,
\[
\left(x^{k-p}(a+x)^{-q}\right)^{(n+1-(p+q))}
=\frac{1}{x^{n+1-q}(a+x)^{n+1-p}}\displaystyle\sum_{s=0}^{n+1-(p+q)}C_{n+k+1-(p+q)-s}x^{n+k+1-(p+q)-s}.
\]
For all $0\leq s<k+1-(p+q),\ k=0,1,\cdots,n$, i.e., the degree of $x$ greater than $n$, the coefficient $C_{n+k+1-(p+q)-s}$ has a factor $\prod_{m=s}^{n-p-q}(k-p-q-m)$, which is equal to zero. Lemma \ref{lem:P} follows.

\noindent{\bf{A.2\ Coefficients of $M(h)$ in \eqref{M3}.}}
The coefficients of $M(h)$ in \eqref{M3} for $n=2$ are as follows. To compare with the averaged function obtained using the averaging method,
we write them as the linear combinations of original parameters $a^+_{ij}$ and $b^+_{ij}$,
instead of $c^+_{ij}$ and $d^+_{ij}$.
\begin{equation}\begin{split}
&\alpha_0=\frac{\pi}{2}\left(2(a^+_{10}+a^-_{10}) -3(b^+_{00}+b^-_{00})+2(b^+_{01}+b^-_{01})\right),\\
&\alpha_1=\frac{\pi}{8}\left(2(a^+_{11}+a^-_{11}) - 9(b^+_{00}+b^-_{00}) + 9(b^+_{01}+b^-_{01})
 - 8(b^+_{02}+b^-_{02}) - 12(b^+_{20}+b^-_{20})\right),\\
&\alpha_2=-\frac{\pi}{16}\left(3(b^+_{00}+b^-_{00}) - 3(b^+_{01}+b^-_{01}) + 3(b^+_{02}+b^-_{02}) +4(b^+_{20}+b^-_{20})\right),\\
&\beta_0=\frac{1}{3}\left(3(a^+_{00}-a^-_{00})-4(a^+_{02}-a^-_{02})- 24(a^+_{20}-a^-_{20})+ 6(b^+_{10}-b^-_{10})\right), \\
&\beta_1= \frac{1}{6}\left(2(a^+_{02}-a^-_{02}) - 3(b^+_{10}-b^-_{10})+ 3(b^+_{11}-b^-_{11})\right), \\
&\gamma_0=2\left (4(a^+_{20}-a^-_{20})-(b^+_{11}-b^-_{11})\right).
\end{split}\end{equation}

\noindent{\bf{A.3\ Proof of Lemma \ref{le:v}.}}
By the definitions of $v(h)$ and \eqref{I02}, it is easy to verify that $v(h)$ satisfies the differential system \eqref{vh}. Obviously, this system has four singularities: two saddles $S_1(-1,1)$ and $S_2(0,0)$, an unstable node $N_1(0,4/5)$ and a stable node $N_2(-1,1/5)$, two invariant lines: $h=-1$ and $h=0$,
and two horizontal isoclines: $v_{\pm}(h)=(2-h\pm\sqrt{4 + h + h^2})/5$.
Definition \eqref{barIk} of $\bar I_k(h)$ shows that the equality $\lim_{h\rightarrow-1^+}v(h)=1$ holds.
Thus, according to the direction of vector field in each region, the graph of $v(h)$ is the stable
manifold of the saddle $S_1$, and it must go to the unstable node $N_1$ as $h$ increases.  It follows that $\lim_{h\rightarrow0^-}v(h)=4/5$.
By a direct computation,
the asymptotic expansions of $\bar I_0$ and $\bar I_2$ near $h=-1$:
\begin{equation}\label{I02-1}
\begin{split}
\bar I_0&=\dfrac{\pi}{4}(h+1)+\dfrac{3\pi}{128}(h+1)^2+o((h+1)^2),\\
\bar I_2&=\dfrac{\pi}{4}(h+1)-\dfrac{\pi}{128}(h+1)^2+o((h+1)^2),
\end{split}
\end{equation}
and when $h\rightarrow0^-$ :
\begin{equation}\label{I02-1}
\begin{split}
\bar I_0&=\frac{2 \sqrt{2}}{3}+\frac{h (-\log(-h)+1+6 \log 2)}{4 \sqrt{2}}+o(h),\\
\bar I_2&=\frac{8 \sqrt{2}}{15}+\frac{h}{\sqrt{2}}-\frac{h^2 (-2 \log (-h)-5+12 \log2)}{64 \sqrt{2}}+o(h^2),
\end{split}
\end{equation}
also shows that the equalities $\lim_{h\rightarrow-1^+}v(h)=1$ and $\lim_{h\rightarrow0^-}v(h)=4/5$ hold.

Next, we prove that $v'(h)<0$ for $h\in(-1,0)$. Note that
\[
v_{-}(-1)=\frac{1}{5},\quad v_{-}(0)=0, \quad v_{+}(-1)=1,\quad v_{-}(0)=\frac{4}{5},
\quad v'_{+}(-1)=-\frac{1}{4}, \quad v'(-1)=-\frac{1}{8}.
\]
We have $v(h)$ is located above $v_+(h)$ and $v_{-}(h)$ near $h\rightarrow-1^+$,
hence it remains above $v_{\pm}(h)$ for $h\in(-1,0)$ by the direction of vector field.
This implies that $v'(h)=(\mathrm{d}v/\mathrm{d}t)/(\mathrm{d}h/\mathrm{d}t)<0$.

\noindent{\bf{A.4\ Two-dimensional Fuchsian systems and the Chebyshev property in \cite{GI}.}}
Consider the functions of the form
\begin{equation}\label{I}
I(h)=p_1(h)I_1(h)+p_2(h)I_2(h), \quad h\in\Sigma,
\end{equation}
where $p_1(h)$ and $p_2(h)$ are polynomials, $I_1(h)$ and $I_2(h)$ are complete Abelian integrals along the ovals
$\gamma(h)$ within a continuous family of ovals
contained in the level sets of a fixed real polynomial $H(x,y)$ (called the
Hamiltonian), and $\Sigma\subset\mathbb{R}$ is the maximal open interval of existence of such ovals.
The vector function $\mathbf{I}(h)=(I_1(h),I_2(h))^{\top}$ satisfies
a two-dimensional first-order Fuchsian system
\begin{equation}\label{I}
\mathbf{I}(h)=\mathbf{A}(h)\mathbf{I}'(h),\quad  '=d/dh,
\end{equation}
with a first-degree polynomial matrix $\mathbf{A(h)}$.

Suppose that
\vspace{-10pt}
\begin{itemize}
\item[(H1)]$\mathbf{A}'$ is a constant matrix having real distinct eigenvalues.\\[-20pt]
\item[(H2)]The equation $\det \mathbf{A}(h)=0$ has real distinct roots $h_0, h_1$ and the identity
trace $\mathbf{A}(h)\equiv(\det \mathbf{A}(h))'$ holds.\\[-20pt]
\item[(H3)]$\mathbf{I}(h)$ is analytic in a neighborhood of $h_0$.
\end{itemize}
\begin{definition}
The real vector space of functions $V$ is said to be Chebyshev in the
complex domain $\mathcal{D}\subset\mathbb{C}$ provided that every function $I\in V\backslash\{0\}$ has at most $\dim V-1$
zeros in $\mathcal{D}$. $V$ is said to be Chebyshev with accuracy $k$ in $\mathcal{D}$ if any function
$I\in V\backslash\{0\}$ has at most $k+\dim V-1$ zeros in $\mathcal{D}$.
\end{definition}
\begin{definition}
Let $I(h), h\in\mathbb{C}$ be a function, locally analytic in a neighborhood of $\infty$
and $s\in \mathbb{R}$. We shall write $I(h)\lesssim h^s$, provided that for every sector $S$ centered at $\infty$
there exists a non-zero constant $C_S$ such that $|I(h)|\leq C_S|h|^s$ for all sufficiently big $|h|, h\in S$.
\end{definition}

For system \eqref{I} satisfying $(\mathrm{H}1)$ and $(\mathrm{H}2)$, the characteristic exponents at infinity
are $-\lambda$ and $-\mu$ where $\lambda'=1/\lambda$ and $\mu'=1/\mu$ are the eigenvalues of the constant
matrix $\mathbf{A}'$ and $\lambda+\mu=2$. Denote $\lambda^*=2$ if $\lambda$ is integer and
$\lambda^*=\max(|\lambda-1|,1-|\lambda-1|)$ otherwise.

Take $s\geq \lambda^*$ and consider the real vector space of functions
\[
V_s=\{I(h)=P(h)I_1(h)+Q(h)I_2(h): P,Q\in \mathbb{R}[h], I(h)\lesssim h^s\},
\]
where $\mathbf{I}(h)=(I_1(h),I_2(h))^{\top}$ is a non-trivial solution of system \eqref{I}, holomorphic in a
neighborhood of $h=h_0$. Then
\begin{theorem}\label{th:GI}
Assume that conditions $(\mathrm{H}1)$-$(\mathrm{H}3)$ hold. If $\lambda\not\in\mathbb{Z}$, then $V_s$ is a Chebyshev
vector space with accuracy $1+[\lambda^*]$ in the complex domain $\mathcal{D}=\mathbb{C}\backslash[h_1,\infty)$. If
$\lambda\in\mathbb{Z}$, then $V_s$ coincides with the space of real polynomials of degree at most $[s]$ which vanish
at $h_0$ and $h_1$.
\end{theorem}

\end{document}